\documentclass[10pt]{amsart}
\usepackage{amsmath}
\usepackage[usenames,dvipsnames]{color}
\usepackage{parskip}
\usepackage{amsfonts}
\usepackage{amscd}
\usepackage[centertags]{amsmath}
\usepackage{amssymb}
\usepackage{stmaryrd}
\usepackage[all,cmtip]{xy}
\usepackage[english]{babel}
\usepackage{tabularx}
\usepackage{amsxtra}
\usepackage{euscript}
\usepackage[T1]{fontenc}
\usepackage{doc, exscale, fontenc, latexsym, syntonly}
\usepackage{amsfonts}
\usepackage{amsthm}
\usepackage{graphicx}
\usepackage{mciteplus}
\usepackage{cite}

\newtheorem{theorem}{Theorem}[section]
\newtheorem{corollary}[theorem]{Corollary}
\newtheorem{lemma}[theorem]{Lemma}
\newtheorem{proposition}[theorem]{Proposition}

\theoremstyle{definition}
\newtheorem{definition}[theorem]{Definition}
\newtheorem{example}[theorem]{Example}
\theoremstyle{remark}

\numberwithin{figure}{section}
\numberwithin{table}{section}

\newcommand*\acknowledgment[1]{%
	\begingroup\noindent
	\rightskip\leftskip
	\begin{flushleft}\textbf{\large Acknowledgment.}\, #1%
		\par\vspace*{1mm}\end{flushleft}\endgroup}
\begin{document}

\title[PROXIMAL MOTION PLANNING ALGORITHMS]{PROXIMAL MOTION PLANNING ALGORITHMS}

\author{MEL\.{I}H \.{I}S and \.{I}SMET KARACA}
\date{\today}

\address{\textsc{Melih Is,}
	Ege University\\
	Faculty of Sciences\\
	Department of Mathematics\\
	Izmir, Türkiye}
\email{melih.is@ege.edu.tr}
\address{\textsc{Ismet Karaca$^1$,}
	Ege University\\
	Faculty of Science\\
	Department of Mathematics\\
	Izmir, Türkiye}
\address{\textsc{Ismet Karaca$^2$,}
	Azerbaijan State Agrarian University\\
	Faculty Agricultural Economics\\
	Department of Agrarian Economics\\
	Gence, Azerbaijan}
\email{ismet.karaca@ege.edu.tr}

\subjclass[2010]{54E17, 55M30, 54H05, 55R80}

\keywords{Proximity, descriptive proximity, nearness theory, motion planning, topological complexity}

\begin{abstract}
    In this paper, we transfer the problem of measuring navigational complexity in topological spaces to the nearness theory. We investigate the most important component of this problem, the topological complexity number (denoted by TC), with its different versions including relative and higher TC, on the proximal Schwarz genus as well as the proximal (higher) homotopic distance. We outline the fundamental properties of some concepts related to the proximal (or descriptive proximal) TC numbers. In addition, we provide some instances of (descriptive) proximity spaces, specifically on basic robot vacuum cleaners, to illustrate the results given on proximal and descriptive proximal TC.
\end{abstract}

\maketitle

\section{Introduction}
\label{introduction}
\quad Farber\cite{Farber:2003} uses algebraic topology techniques to first understand studies of motion planning problem-solving. This new interpretation makes the notion of topological complexity (denoted by TC), a homotopy invariant, the main focus of the motion planning problem. Rudyak\cite{Rudyak:2010}, Pavesic\cite{Pavesic:2019}, and İs and Karaca\cite{MelihKaraca:2022}, respectively, transform the problem of computing the TC number of a path-connected space to the problems of computing the higher dimensional TC number of a path-connected space, computing the TC number of a fibration, and computing the higher dimensional TC number of a fibration. Although the parametrized topological complexity number is the main focus of recent studies on TC, many other TC-related concepts, such as LS-category (denoted by cat) and homotopic distance, and many topological spaces, particularly manifolds, and CW-complexes, are also thoroughly investigated. 

\quad The problem of finding the TC number is not limited to topological spaces, and there is a lot of work on digital images (see \cite{KaracaMelih:2018,MelihKaraca:2020,MelihKaraca:2022(1),MelihKaraca:2022(2),MelihKaraca:2023(2)} for the TC numbers of digital images). In fact, the TC-related results on digital images are compared with the TC-related results in topological spaces, and the differences are clearly stated \cite{MelihKaraca:2022(2)}. In this sense, TC investigations on nearness theory offer a fresh perspective on the motion planning problem and help to establish precise comparisons in both topological spaces and digital images. Just as observed on digital images, the differences in results for TC computations in nearness theory have the potential to directly influence the perspective on the robot motion planning problem of topological spaces. The idea of nearness theory is derived from relations such as proximities (or descriptive proximities), which express that two sets are near (or descriptively near) each other. The strong structure of proximity spaces stands out in the variety of application areas such as cell biology, visual merchandising, or microscope images\cite{NaimpallyPeters:2013} (see \cite{NaimpallyPeters:2013,PetersNaimpally:2012} for more interesting examples). This work aims to start a new discussion by adding explorations on TC to create motion planning algorithms for areas involving proximity spaces.

\quad The paper is organized as follows. Basic concepts and facts for the proximity and the descriptive proximity are covered in Section \ref{sec:preliminaries}. The first definition of the proximal TC number is offered using the notion of a proximal Schwarz genus, which is initially introduced in Section \ref{sec:1}. The variation in the number of TC according to the coarseness/fineness of a proximity space is revealed with accurate findings after the initial examples of the TC computation are shown. As a type of TC number, the version of the relative TC number in proximity spaces is introduced. Proximal homotopic distance and its characteristics are discussed in Section \ref{sec:2} as another way of identifying proximal TC numbers. A change in proximal homotopic distance is observed with respect to the coarseness/fineness of a proximity space in just the same way as TC. Eventually, in nearness theory, the concept of TC is given an alternative definition via the proximal homotopic distance. In Section \ref{sec:3}, higher dimensional versions of both the proximal homotopic distance and the TC number are introduced, and their basic properties are also discussed in the higher dimensions. In Section \ref{sec:4}, descriptive proximity interpretations of the Schwarz genus, homotopic distance, TC, relative TC, and generalizations of all these numbers are exhibited. Some examples of descriptive proximal TC are presented. Section \ref{sec:conclusion} is devoted to some open problems and recommendations. 

\section{Preliminaries}
\label{sec:preliminaries}

\quad Basic facts about proximity and descriptive proximity spaces, particularly proximal and descriptive proximal homotopy theory, should be remembered. These provide a strong ground to construct the notion of TC on the nearness (or descriptive nearness) theory.

\quad Let $d$ be a pseudo-metric on a set $X$, and a binary relation $\delta$ be defined by
\begin{eqnarray*}
	''E \delta F \ \ \Leftrightarrow \ \ D(E,F) = 0'',
\end{eqnarray*}
where $D(E,F)$ is the set $\inf\{d(x_{0},x_{1}): x_{0} \in E, x_{1} \in F\}$. Then $\delta$ is called an \textit{Efremovic proximity} on $X$ (or simply \textit{proximity} on $X$) if it satisfies 
\begin{eqnarray*}
	&& E \delta F \ \ \Rightarrow \ \ F \delta E, \hspace*{1.0cm} (E \cup F) \delta G \ \ \Leftrightarrow \ \ E \delta G \ \vee \ F \delta G, \\
	&& E \delta F \ \ \Rightarrow \ \ E \neq \emptyset \ \land \ F \neq \emptyset, \hspace*{1.0cm} E \cap F \neq \emptyset \ \ \Rightarrow \ \ E \delta F,\\
	&& E \underline{\delta} F \ \ \Rightarrow \ \ \exists G \subset X: E \underline{\delta} G \ \land \ (X-G) \underline{\delta} F,
\end{eqnarray*}
where $E \delta F$ means that \textit{“$E$ is near $F$''} \cite{Efremovic:1952,Smirnov:1952,NaimpallyWarrack:1970}. Also, the notation $E \underline{\delta} F$ is read as \textit{“$E$ is far from $F$''}. The pair $(X,\delta)$ is simply called a \textit{proximity space}.

\quad Given a subset $E$ of $(X,\delta)$, $E \in \mathcal{K}$, i.e., $E$ is closed in $X$ if and only if the fact $x$ is near $E$ implies that $x \in E$. This provides the fact that for a proximity $\delta$ on $X$, one has a topology $\tau(\delta)$ on $X$ induced by $\delta$ via Kuratowski closure operator\cite{NaimpallyWarrack:1970}. Mathematically, given a proximity $\delta$ and a topology $\tau(\delta)$ on a set $X$, the closure of $E$, denoted by $\overline{E}$, coincides with the set of points in $X$ that is near $E$ \cite{NaimpallyWarrack:1970}.  

\quad Let $E$ be a subset of a proximity space $(X,\delta)$. Assume that $\delta_{E}$ is given by $F \delta G$ if and only if $F \delta_{E} G$ for $F$, $G \in 2^{E}$. Then $(E,\delta_{E})$ is called a \textit{subspace} of $(X,\delta)$ \cite{NaimpallyWarrack:1970}. In addition, $\tau(\delta_{E})$ is said to be a \textit{subspace topology on $E$} by $\tau(\delta)$. Given two proximity spaces $(X,\delta)$ and $(X^{'},\delta^{'})$, it is possible to construct new proximity on their cartesian product proximity space $X \times X^{'}$\cite{Leader:1964}: Let $E_{1} \times E_{2}$ and $F_{1} \times F_{2}$ be a subset of $X \times X^{'}$. Then $E_{1} \times E_{2}$ \textit{is near} $F_{1} \times F_{2}$ provided that $E_{1}$ is near $F_{1}$ with respect to $\delta$ and $E_{2}$ is near $F_{2}$ with respect to $\delta^{'}$. Let $(X,\delta)$ and $(X,\delta^{'})$ be two proximity spaces. Then $\delta$ is \textit{finer than} $\delta^{'}$ (or $\delta^{'}$ is \textit{coarser than} $\delta$), denoted by $\delta > \delta^{'}$, if and only if $E \delta F$ requires $E \delta^{'} F$ \cite{NaimpallyWarrack:1970}. 

\quad Let us consider the descriptive proximity case and its basic facts. Let $X$ be a nonempty set, $x \in X$ any point, $E \subset X$ any subset, and $e \in E$ a point. The map $\phi_{j} : X \rightarrow \mathbb{R}$ expresses as a probe function for each $j = 1,\cdots,m$, and a feature value of $x$ is denoted by $\phi_{j}(x)$. Consider $\Phi$ as the set of probe functions $\{\phi_{1},\cdots,\phi_{m}\}$. Then $\mathcal{Q}(E) = \{\Phi(e) : e \in E\}$ is the set of descriptions of the point $e$ for a feature vector $\Phi(e)$ given by $(\phi_{1}(e),\cdots,\phi_{m}(e))$. Let $E$, $F \subset X$. A binary relation $\delta_{\Phi}$ on $X$ is defined by 
\begin{eqnarray}
	E \delta_{\Phi} F \ \Leftrightarrow \ \mathcal{Q}(E) \cap \mathcal{Q}(F) \neq \emptyset,
\end{eqnarray}
and “$E$ is \textit{descriptively near} $F$'' means that $E \delta_{\Phi} F$ \cite{Peters1:2007,Peters2:2007,Peters:2013}. Also, $E \underline{\delta_{\Phi}} F$ is read as “$E$ is \textit{descriptively far from} $F$''. The intersection and the union for descriptive proximity is different from the proximity \cite{Peters:2013}: The \textit{descriptive intersection of $E$ and $F$} is 
\begin{eqnarray*}
	E \displaystyle \bigcap_{\Phi} F = \{x \in E \cup F : \Phi(x) \in \mathcal{Q}(E) \ \land \ \Phi(x) \in \mathcal{Q}(F)\},
\end{eqnarray*}
and the \textit{descriptive union of $E$ and $F$} is
\begin{eqnarray*}
	E \displaystyle \bigcup_{\Phi} F = \{x \in E \cup F : \Phi(x) \in \mathcal{Q}(E) \ \vee  \ \Phi(x) \in \mathcal{Q}(F)\}.
\end{eqnarray*}

$\delta_{\Phi}$ defined by (1) is said to be a \textit{descriptive Efremovic proximity} on $X$ (or simply \textit{descriptive proximity} on $X$) if it satisfies
\begin{eqnarray*}
	&& E \delta_{\Phi} F \ \ \Rightarrow \ \ F \neq \emptyset \ \land E \neq \emptyset, \hspace*{1.0cm} E \displaystyle \bigcap_{\Phi} F \neq \emptyset \ \ \Rightarrow \ \ F \delta_{\Phi} E,\\
	&& E \displaystyle \bigcap_{\Phi} F \neq \emptyset \ \ \Rightarrow \ \ F \displaystyle \bigcap_{\Phi} E, \hspace*{1.0cm} E \delta_{\Phi} (F \cup G) \ \ \Leftrightarrow \ \ E \delta_{\Phi} F \ \vee \ E \delta_{\Phi} G,\\
	&& E \underline{\delta_{\Phi}} F \ \ \Rightarrow \ \ \exists G \subset X: E \underline{\delta_{\Phi}} G \ \land \ (X-G) \underline{\delta_{\Phi}} F,
\end{eqnarray*}
and the pair $(X,\delta_{\Phi})$ is simply called a \textit{descriptive proximity space} \cite{NaimpallyPeters:2013}.

\quad Given a descriptive proximity $\delta_{\Phi}$ on $X$, a topology $\tau(\delta_{\Phi})$ is induced by $\delta_{\Phi}$ on $X$. Let $E$ be a subset of a set $X$ with the descriptive proximity $\delta_{\Phi}$. Define $\delta_{\Phi_{E}}$ by $F \delta_{\Phi} G$ if and only if $F \delta_{\Phi_{E}} G$ for $F$, $G \in 2^{E}$. Then $(E,\delta_{\Phi_{E}})$ is called a \textit{subspace} of $(X,\delta_{\Phi})$. Given two descriptive proximity spaces $(X,\delta_{\Phi})$ and $(X^{'},\delta_{\Phi}^{'})$, one can construct a new proximity on their cartesian product descriptive proximity space $X \times X^{'}$: Let $E_{1} \times E_{2}$ and $F_{1} \times F_{2}$ be a subset of $X \times X^{'}$. Then $E_{1} \times E_{2}$ \textit{is descriptively near} $F_{1} \times F_{2}$ provided that $E_{1}$ is descriptively near $F_{1}$ with respect to $\delta_{\Phi}$ and $E_{2}$ is descriptively near $F_{2}$ with respect to $\delta_{\Phi}^{'}$. 

\subsection{Proximal Homotopy and Descriptive Proximal Homotopy}
\label{subsec:subsec2}

A map $f : (X,\delta) \rightarrow (X,\delta^{'})$ is called \textit{proximally continuous} if $E \delta F$ implies that $f(E) \delta^{'} f(F)$ for $E$, $F \in 2^{X}$\cite{Efremovic:1952,Smirnov:1952}. A proximally continuous map is simply denoted by “pc-map'' in this manuscript. If $f$ is bijective, pc-map, and $f^{-1}$ is pc-map, then $f$ is a \textit{proximity isomorphism}\cite{NaimpallyWarrack:1970}. Moreover, $(X,\delta)$ and $(X^{'},\delta^{'})$ are called \textit{proximally isomorphic spaces}. A property is a \textit{proximity invariant} provided that it is preserved under proximity isomorphisms \cite{NaimpallyWarrack:1970}. Recall that any pc-map $f : (X,\delta) \rightarrow (X,\delta^{'})$ admits a continuous map $f : (X,\tau(\delta)) \rightarrow (X^{'},\tau(\delta^{'}))$ on topological spaces. This shows that a topological invariant is a proximity invariant, but the converse need not be true.

\begin{lemma}\cite{PetersTane:2021}(Gluing Lemma for Proximity)
	Assume that the maps \linebreak$f_{1} : (X,\delta_{1}) \rightarrow (X^{'},\delta^{'})$ and $f_{2} : (Y,\delta_{2}) \rightarrow (X^{'},\delta^{'})$ are two pc-maps such that $f_{1}$ and $f_{2}$ agree on $X \cap Y$. Assume that the map $f_{1} \cup f_{2} : (X \cup Y,\delta) \rightarrow (X^{'},\delta^{'})$ is given by \[f_{1} \cup f_{2}(z) = \begin{cases}
		f_{1}(z), & z \in X \\
		f_{2}(z), & z \in Y
	\end{cases}\] for all $z \in X \cup Y$. Then $f_{1} \cup f_{2}$ is a pc-map.
\end{lemma}

\begin{definition}\cite{PetersTane:2021}
	Let $f_{1}$ and $f_{2}$ be a map from $(X,\delta)$ to $(X^{'},\delta^{'})$. Then it is said to be that $f_{1}$ is \textit{proximally homotopic to} $f_{2}$ if there exists a pc-map $G$ from $X \times I$ to $X^{'}$ for which $G(x,0) = f_{1}(x)$ and $G(x,1) = f_{2}(x)$. The pc-map $G : (X \times I,\delta_{1}) \rightarrow (X^{'},\delta^{'})$ satisfying two properties $G(x,0) = f_{1}(x)$ and $G(x,1) = f_{2}(x)$ is called a \textit{proximal homotopy between $f_{1}$ and $f_{2}$}. 
\end{definition}

\quad An equivalence relation on proximity spaces is the proximal homotopy. $(X,\delta)$ is called \textit{proximally contractible} \cite{PetersTane:2021} if $1_{X}: (X,\delta) \rightarrow (X,\delta)$ is proximally homotopic to $c : (X,\delta) \rightarrow (X,\delta)$, $c(x) = x_{0}$, where $x_{0}$ is a constant in $X$. 

\quad Let $(X,\delta)$ be a proximity space and $x_{0}$, $x_{1} \in X$ any points. Then a pc-map $g : [0,1] \rightarrow X$ with $g(0) = x_{0}$ and $g(1) = x_{1}$ defines a \textit{proximal path with the endpoints $x_{0}$ and $x_{1}$} in $X$\cite{PetersTane:2021}. $(X,\delta)$ is said to be a \textit{path-connected proximity space} provided that there is a proximal path with the endpoints $x_{0}$ and $x_{1}$ in $X$ for any $x_{0}$, $x_{1} \in X$\cite{MelihKaraca:2023}. $(X,\delta)$ is said to be a \textit{connected proximity space} provided that $E$ is near $F$ if the union of $E$ and $F$ gives us $X$ for any nonempty $E$, $F \subset X$\cite{MrowkaPervin:1964}.

\begin{theorem}\cite{MelihKaraca:2023}
	A proximity space is path-connected if and only if it is connected.
\end{theorem}

\quad Let us recall basic facts related to the descriptive proximal homotopy. A map $f : (X,\delta_{\Phi}) \rightarrow (X,\delta_{\Phi}^{'})$ is called \textit{descriptive proximally continuous} if $E \delta_{\Phi} F$ implies that $f(E) \delta_{\Phi}^{'} f(F)$ for $E$, $F \in 2^{X}$\cite{Peters:2014,PetersTane:2021}. A descriptive proximally continuous map is simply denoted by “dpc-map'' in this manuscript. If $f$ is bijective, a dpc-map, and $f^{-1}$ is a dpc-map, then $f$ is a \textit{descriptive proximity isomorphism}\cite{NaimpallyWarrack:1970}. Moreover, $(X,\delta_{\Phi})$ and $(X^{'},\delta_{\Phi}^{'})$ are called \textit{descriptive proximally isomorphic spaces}. A property is a \textit{descriptive proximity invariant} provided that it is preserved under descriptive proximity isomorphisms \cite{NaimpallyWarrack:1970}.

\begin{lemma}\cite{PetersTane:2021}(Gluing Lemma for Descriptive Proximity)
	Given two dpc-maps $f_{1} : (X,\delta_{\Phi}^{1}) \rightarrow (X^{'},\delta_{\Phi}^{'})$ and $f_{2} : (Y,\delta_{\Phi}^{2}) \rightarrow (X^{'},\delta_{\Phi}^{'})$ such that $f_{1}$ and $f_{2}$ agree on $X \cap Y$. Assume that the map $f_{1} \cup f_{2} : (X \cup Y,\delta_{\Phi}) \rightarrow (X^{'},\delta_{\Phi}^{'})$ is given by $f_{1} \cup f_{2}(z) = \begin{cases}
		f_{1}(z), & z \in X \\
		f_{2}(z), & z \in Y
	\end{cases}$ for all $z \in X \cup Y$. Then $f_{1} \cup f_{2}$ is a dpc-map.
\end{lemma}

\begin{definition}\cite{PetersTane:2021}
	Let $f_{1}$ and $f_{2}$ be a map from $(X,\delta_{\Phi})$ to $(X^{'},\delta_{\Phi}^{'})$. Then it is said to be that $f_{1}$ is \textit{descriptive proximally homotopic to} $f_{2}$ if there exists a dpc-map $G$ from $X \times I$ to $X^{'}$ for which $G(x,0) = f_{1}(x)$ and $G(x,1) = f_{2}(x)$. The dpc-map $G : (X \times I,\delta_{\Phi}^{1}) \rightarrow (X^{'},\delta_{\Phi}^{'})$ satisfying two conditions $G(x,0) = f_{1}(x)$ and $G(x,1) = f_{2}(x)$ is called a \textit{descriptive proximal homotopy between $f_{1}$ and $f_{2}$}.
\end{definition}

\quad An equivalence relation on descriptive proximity spaces is the descriptive proximal homotopy. $(X,\delta_{\Phi})$ is called \textit{descriptive proximally contractible} \cite{PetersTane:2021} if $1_{X}$ from $(X,\delta_{\Phi})$ to $(X,\delta_{\Phi})$ is descriptive proximally homotopic to $c : (X,\delta_{\Phi}) \rightarrow (X,\delta_{\Phi})$, $c(x) = x_{0}$, where $x_{0}$ is a constant in $X$. 

\quad Given a descriptive proximity space $(X,\delta_{\Phi})$ and any points $x_{0}$, $x_{1} \in X$, a dpc-map $g : [0,1] \rightarrow X$ with $g(0) = x_{0}$ and $g(1) = x_{1}$ defines a \textit{descriptive proximal path with the endpoints $x_{0}$ and $x_{1}$} in $X$\cite{PetersTane:2021}. $(X,\delta_{\Phi})$ is said to be a \textit{path-connected descriptive proximity space} provided that there is a descriptive proximal path with the endpoints $x_{0}$ and $x_{1}$ in $X$ for any $x_{0}$, $x_{1} \in X$\cite{MelihKaraca:2023}. $(X,\delta_{\Phi})$ is said to be a \textit{connected descriptive proximity space} provided that $E$ is descriptively near $F$ if $E \cup F = X$ for any nonempty $E$, $F \subset X$\cite{MrowkaPervin:1964}.

\begin{theorem}\cite{MelihKaraca:2023}
	A descriptive proximity space is path-connected if and only if it is connected.
\end{theorem}

\subsection{Proximal Fibration and Descriptive Proximal Fibration}
\label{subsec:subsec3}

\quad Assume that $(X,\delta_{1})$ and $(Y,\delta_{2})$ are two proximal spaces. Then one can construct a proximity relation $\delta$ on the \textit{proximal mapping space} \[Y^{X} = \{\alpha : X \rightarrow Y \ | \ \alpha \ \text{is a pc-map}\}\] as follows\cite{PeiRen:1985,MelihKaraca:2023}: For any $\{\alpha_{i}\}_{i \in I}$, $\{\beta_{j}\}_{j \in J} \in Y^{X}$, $\{\alpha_{i}\}_{i \in I}$ is near $\{\beta_{j}\}_{j \in J}$ with respect to $\delta$ provided that $\alpha_{i}(E)$ is near $\beta_{j}(F)$ with respect to $\delta_{2}$ when $E$ is near $F$ with respect to $\delta_{1}$ for $E$, $F \in 2^{X}$. 

\quad A map $F : (X,\delta_{1}) \rightarrow (Z^{Y},\delta^{'})$ is a pc-map if the fact $E$ is near $F$ with respect to $\delta_{1}$ implies that $F(E)$ is near $F(F)$ with respect to $\delta^{'}$ for any $E$, $F \in 2^{X}$.

\begin{theorem}\cite{MelihKaraca:2023}
	Given three proximity spaces $(X,\delta_{1})$, $(Y,\delta_{2})$, and $(Z,\delta_{3})$, we have that $\alpha : (Z^{X \times Y},\delta_{4}) \rightarrow ((Z^{Y})^{X},\delta_{5})$ and $\beta : ((Y \times Z)^{X},\delta_{6}) \rightarrow (Z^{X},\delta_{7})$ are proximity isomorphisms. 
\end{theorem}

\quad Given two proximity spaces $(X,\delta_{1})$ and $(Y,\delta_{2})$, $e_{X,Y} : (Y^{X} \times X,\delta) \rightarrow (Y,\delta_{2})$, $e(\alpha,x) = \alpha(x)$ is called the \textit{proximal evaluation map}, and it is a pc-map\cite{MelihKaraca:2023}. Proximal fiber bundles and proximal fibrations are introduced in \cite{PetersTane:2022}, and the latter is improved in \cite{MelihKaraca:2023}.

\begin{definition}\cite{PetersTane:2022,MelihKaraca:2023}
	Let $p : (X,\delta) \rightarrow (X^{'},\delta^{'})$ be a pc-map. Given any pc-map $f : (X^{''},\delta^{''}) \rightarrow (X,\delta)$ and any proximal homotopy $H : (X^{''} \times I,\delta_{1}) \rightarrow (X^{'},\delta^{'})$ such that $p \circ f = H \circ i_{0}$, $p$ has the \textit{proximal homotopy lifting property (PHLP) with respect to $(X^{''},\delta^{''})$} provided that there is a proximal homotopy $H^{'} : (X^{''} \times I,\delta_{1}) \rightarrow (X,\delta)$ with $H^{'}(x^{''},0) = f(x^{''})$ and $p \circ H^{'}(x^{''},t) = H(x^{''},t)$.
	\begin{displaymath}
		\xymatrix{
			X^{''} \ar[r]^{f} \ar[d]_{i_{0}} &
			X \ar[d]^{p} \\
			X'' \times I \ar[r]_{H} \ar@{.>}[ur]^{H^{'}} & X^{'}. }
	\end{displaymath}
\end{definition}

\begin{definition}\cite{PetersTane:2022,MelihKaraca:2023}
	Let $p : (X,\delta) \rightarrow (X^{'},\delta^{'})$ be a pc-map. If it has the PHLP for any $(X^{''},\delta^{''})$, then it is called a \textit{proximal fibration}.
\end{definition}

\quad Each of two pc-maps $p_{1} : (X^{I},\delta^{'}) \rightarrow (X,\delta_{1})$  and $p_{2} : (X^{I},\delta^{'}) \rightarrow (X \times X,\delta^{'})$, defined by $p_{1}(\alpha) = \alpha(0)$ and $p_{2}(\alpha) = (\alpha(0),\alpha(1))$, respectively, is one of examples of a proximal fibration. They are called \textit{proximal path fibrations}\cite{MelihKaraca:2023}. Proximal fibrations have useful properties as follows\cite{MelihKaraca:2023}: Let $p_{1}$ and $p_{2}$ be two proximal fibration, $f$ a pc-map. Then $p_{1} \circ p_{2}$, $p_{1} \times p_{2}$, and the pullback $f^{\ast}p_{1}$ of the proximal fibration $p_{1}$ are proximal fibrations as well. Let $(Z,\delta_{3})$ be a proximity space and $p : (X,\delta_{1}) \rightarrow (Y,\delta_{2})$ a proximal fibration. Then the map $p_{\ast} : (X^{Z},\delta_{3}^{'}) \rightarrow (Y^{Z},\delta_{3}^{''})$ is also a proximal fibration.

\quad Consider the descriptive case of proximal mapping spaces and proximal fibrations. Given two descriptive proximity spaces $(X,\delta_{\Phi}^{1})$ and $(Y,\delta_{\Phi}^{2})$, a descriptive proximity relation $\delta_{\Phi}$ on the \textit{descriptive proximal mapping space} \[Y^{X} = \{\alpha : X \rightarrow Y \ | \ \alpha \ \text{is a dpc-map}\}\] is defined as follows\cite{PeiRen:1985,MelihKaraca:2023}: For any $\{\alpha_{i}\}_{i \in I}$, $\{\beta_{j}\}_{j \in J} \in Y^{X}$, $\{\alpha_{i}\}_{i \in I}$ is descriptively near $\{\beta_{j}\}_{j \in J}$ with respect to $\delta_{\Phi}$ provided that $\alpha_{i}(E)$ is descriptively near $\beta_{j}(F)$ with respect to $\delta_{\Phi}^{2}$ when $E$ is descriptively near $F$ with respect to $\delta_{\Phi}^{1}$ for $E$, $F \in 2^{X}$. 

\quad A map $f : (X,\delta_{\Phi}^{1}) \rightarrow (Z^{Y},\delta_{\Phi}^{'})$ is a dpc-map if the fact $E$ is descriptively near $F$ with respect to $\delta_{\Phi}^{1}$ implies that $f(E)$ is descriptively near $f(F)$ with respect to $\delta_{\Phi}^{'}$ for any $E$, $F \in 2^{X}$.

\begin{theorem}\cite{MelihKaraca:2023}
	Given three descriptive proximity spaces $(X,\delta_{\Phi}^{1})$, $(Y,\delta_{\Phi}^{2})$, and $(Z,\delta_{\Phi}^{3})$, $\alpha : (Z^{X \times Y},\delta_{\Phi}^{4}) \rightarrow ((Z^{Y})^{X},\delta_{\Phi}^{5})$ and $\beta : ((Y \times Z)^{X},\delta_{\Phi}^{6}) \rightarrow (Z^{X},\delta_{\Phi}^{7})$ are descriptive proximity isomorphisms. 
\end{theorem}

\quad Given two descriptive proximity spaces $(X,\delta_{\Phi}^{1})$ and $(Y,\delta_{\Phi}^{2})$, \[e_{X,Y} : (Y^{X} \times X,\delta_{\Phi}) \rightarrow (Y,\delta_{\Phi}^{2}),\]
defined by $e(\alpha,x) = \alpha(x)$, is called the \textit{descriptive proximal evaluation map}, and it is a dpc-map\cite{MelihKaraca:2023}. Descriptive proximal fiber bundles and descriptive proximal fibrations are introduced in \cite{PetersTane:2022}, and the latter is improved in \cite{MelihKaraca:2023}.

\begin{definition}\cite{PetersTane:2022,MelihKaraca:2023}
	Let $p : (X,\delta_{\Phi}) \rightarrow (X^{'},\delta_{\Phi}^{'})$ be a dpc-map. Given any dpc-map $f : (X^{''},\delta_{\Phi}^{''}) \rightarrow (X,\delta_{\Phi})$ and any descriptive proximal homotopy \linebreak$H : (X^{''} \times I,\delta_{\Phi}^{1}) \rightarrow (X^{'},\delta_{\Phi}^{'})$ such that $p \circ f = H \circ i_{0}$, $p$ has the \textit{descriptive proximal homotopy lifting property (DPHLP) with respect to $(X^{''},\delta_{\Phi}^{''})$} provided that there is a descriptive proximal homotopy $H^{'} : (X^{''} \times I,\delta_{\Phi}^{1}) \rightarrow (X,\delta_{\Phi})$ with $H^{'}(x^{''},0) = f(x^{''})$ and $p \circ H^{'}(x^{''},t) = H(x^{''},t)$.
	\begin{displaymath}
		\xymatrix{
			X^{''} \ar[r]^{f} \ar[d]_{i_{0}} &
			X \ar[d]^{p} \\
			X'' \times I \ar[r]_{H} \ar@{.>}[ur]^{H^{'}} & X^{'}. }
	\end{displaymath}
\end{definition}

\begin{definition}\cite{PetersTane:2022,MelihKaraca:2023}
	Let $p : (X,\delta_{\Phi}) \rightarrow (X^{'},\delta_{\Phi}^{'})$ be a dpc-map. If it has the DPHLP for any $(X^{''},\delta_{\Phi}^{''})$, then it is called a \textit{descriptive proximal fibration}.
\end{definition}

\quad Each of dpc-maps \[p_{1} : (X^{I},\delta_{\Phi}^{'}) \rightarrow (X,\delta_{\Phi}^{1})\] and \[p_{2} : (X^{I},\delta_{\Phi}^{'}) \rightarrow (X \times X,\delta_{\Phi}^{'}),\] defined by $p_{1}(\alpha) = \alpha(0)$ and $p_{2}(\alpha) = (\alpha(0),\alpha(1))$, respectively, is one of examples of a descriptive proximal fibration. They are called \textit{descriptive proximal path fibrations}\cite{MelihKaraca:2023}. Descriptive proximal fibrations have useful properties as follows\cite{MelihKaraca:2023}: Let $p_{1}$ and $p_{2}$ be two descriptive proximal fibration, $f$ a dpc-map. Then $p_{1} \circ p_{2}$, \linebreak$p_{1} \times p_{2}$, and the pullback $f^{\ast}p_{1}$ of the descriptive proximal fibration $p_{1}$ are descriptive proximal fibrations as well. For any descriptive proximity space $(Z,\delta_{3})$, the map $p_{\ast} : (X^{Z},\delta_{\Phi}^{3'}) \rightarrow (Y^{Z},\delta_{\Phi}^{3''})$ is a descriptive proximal fibration. Let $(Z,\delta_{\Phi}^{3})$ be a descriptive proximity space and $p : (X,\delta_{\Phi}^{1}) \rightarrow (Y,\delta_{\Phi}^{2})$ a descriptive proximal fibration. Then $p_{\ast} : (X^{Z},\delta_{\Phi}^{3'}) \rightarrow (Y^{Z},\delta_{\Phi}^{3''})$ is also a descriptive proximal fibration.

\section{Proximal TC Numbers By Using Proximal Schwarz Genus}
\label{sec:1}

\quad Given two proximal paths $\alpha_{1} : [0,1] \rightarrow X$ and $\alpha_{2} : [0,1] \rightarrow X$ in $X$, they are said to be near provided that for any $E$, $F \in 2^{I}$, $E \delta^{'} F$ implies that $\alpha_{1}(E) \delta \alpha_{2}(F)$, where $\delta^{'}$ is a proximity on $I$ and $\delta$ is a proximity on $X$. The map $\pi : PX \rightarrow X \times X$ with $\pi(\alpha) = (\alpha(0),\alpha(1))$ is called a proximal path fibration,  where $PX$ is defined by the set $\{\alpha \ | \ \alpha : [0,1] \rightarrow X \ \ \text{is a proximal path}\}$.

\begin{definition}
	Let $p : (X,\delta) \rightarrow (X^{'},\delta^{'})$ be a proximal fibration. Then the proximal Schwarz genus of $p$ is the possible minimum integer $m > 0$ if $X^{'}$ has a cover $\{U_{1},U_{2},\cdots,U_{m}\}$, i.e., $X^{'}$ can be written as the union of $U_{1}$, $U_{2}$, $\cdots$, $U_{m}$, such that there exists a pc-map $s_{i} : U_{i} \rightarrow X$ with the property $p \circ s_{i} = 1_{U_{i}}$ for each $i \in \{1,\cdots,m\}$.
\end{definition}

The proximal Schwarz genus of $p : (X,\delta) \rightarrow (X^{'},\delta^{'})$ is denoted by genus$_{\delta,\delta^{'}}(p)$.

\begin{definition}
	Let $(X,\delta)$ be a connected proximity space and the pc-map \linebreak$\pi : PX \rightarrow X \times X$, $\pi(\alpha) = (\alpha(0),\alpha(1))$, a proximal path fibration. Then the proximal topological complexity of $X$ (or simply proximal complexity), denoted by TC$(X,\delta)$, is the proximal Schwarz genus of $\pi$.
\end{definition}

\quad Throughout the article, we note that the proximity space $X$ is considered connected whenever the topological complexity number is computed. Since TC is a homotopy invariant \cite{Farber:2003}, the proximal topological complexity is a proximal homotopy invariant. Indeed, a topological invariant is a proximity invariant \cite{NaimpallyWarrack:1970}. 

\begin{example}
	Let $([1,2],\delta)$ be a proximity space. Define the map \[s : [1,2] \times [1,2] \rightarrow P[1,2]\] as follows. For any near subsets $E_{1} \times E_{2}$ and $F_{1} \times F_{2}$ in $[1,2] \times [1,2]$, $s(E_{1},E_{2})$ and $s(F_{1},F_{2})$ are near. It follows that for a point $(E_{1},E_{1})$, $s(E_{1},E_{2})$ is a (proximally continuous) proximal path between $E_{1}$ and $E_{2}$. For a proximal fibration \[\pi : P[1,2] \rightarrow [1,2] \times [1,2]\] with $\pi(\alpha) = (\alpha(0),\alpha(1))$, we find that $\pi \circ s$ is the identity map. This shows that TC$([1,2],\delta) = 1$.
\end{example}

\begin{example}
	Let $\delta$ be a discrete proximity on $X$. Then we have that
	\begin{eqnarray*}
		E \delta F \ \ \Leftrightarrow \ \ E \cap F \neq \emptyset.
	\end{eqnarray*}
    We first show that if $(X,\delta)$ is a discrete proximity space, then it is not a connected proximity space. Suppose that $(X,\delta)$ is a connected proximity space. Then for any subset $E \in 2^{X}$, $E \delta (X-E)$ since $E \cup (X-E) = X$. On the other hand, $E \cap (X-E) = \emptyset$ implies that $E \underline{\delta} (X-E)$. Thus, we have a contradiction. As a result, TC$(X,\delta)$ cannot be computed when $\delta$ is a discrete proximity on $X$.
\end{example}

\begin{theorem}\label{teo1}
	Let $(X,\delta)$ be a proximity space. Then $X$ is proximally contractible if and only if there exists a pc-map (proximally continuous motion planning) \linebreak$s : X \times X \rightarrow PX$.
\end{theorem}

\begin{proof}
	Let $H : X \times I \rightarrow X$ be a proximal homotopy with $H(x,0) = a_{0}$ and $H(x,1) = x$ for any $x \in X$, where $a_{0}$ is a fixed point in $X$. Assume that $a$ and $b$ are any points in $X$. Then we can construct a proximal path from $a$ to $b$ as follows. Consider the proximal paths $\alpha : [0,1] \rightarrow X$ with $\alpha(t) = H(a,t)$ (from $a_{0}$ to $a$) and $\beta : [0,1] \rightarrow X$ with $\beta(t) = H(b,t)$ (from $a_{0}$ to $b$). The composition of $\beta$ and $\alpha^{-1}$ gives a proximal path from $a$ to $b$. Since the composition of two pc-maps is continuous, it follows that for any points $a$ and $b$ in $X$, there is a pc-map from $a$ to $b$ by using the proximally contraction homotopy.
	
	Conversely, let $s : X \times X \rightarrow PX$ be a pc-map. Define the proximal map
	\begin{eqnarray*}
		&&H : X \times I \longrightarrow X \\
		&&\hspace*{0.7cm} (x,t) \longmapsto H(x,t) = s(a_{0},x)(t)
	\end{eqnarray*}
for $x \in X$ and $t \in I$. Since $s$ is a pc-map, $H$ is a pc-map. We also find that
\begin{eqnarray*}
	H(x,0) = s(a_{0},x)(0) = a_{0} \ \ : \ \ \text{constant map}
\end{eqnarray*}
and
\begin{eqnarray*}
	H(x,1) = s(a_{0},x)(1) = x = 1_{X}(x) \ \ : \ \ \text{identity map}.
\end{eqnarray*}
Thus, $H$ is a proximal homotopy between the constant map and the identity map. As a result, $X$ is proximally contractible.
\end{proof}

\quad We immediately have the following result from Theorem \ref{teo1}.
\begin{corollary}\label{cor1}
	Let $(X,\delta)$ be a proximity space. Then TC$(X,\delta) = 1$ if and only if $X$ is proximally contractible.
\end{corollary}

\begin{example}
	Let $\delta^{'}$ be an indiscrete proximity on $X^{'}$. Then we have that
	\begin{eqnarray*}
		E \delta^{'} F \ \ \Leftrightarrow \ \ E \neq \emptyset \ \ \text{and} \ \ F \neq \emptyset.
	\end{eqnarray*}
	We first show that every map whose codomain has an indiscrete proximity is a pc-map. Let $(X,\delta)$ be any proximity space. Given a function $f : (X,\delta) \rightarrow (X^{'},\delta^{'})$ for an indiscrete proximity $\delta^{'}$, $E \delta F$ implies that $f(E) \delta^{'} f(F)$ for any nonempty subsets $E$, $F$ in $2^{X}$. Indeed, $E \neq \emptyset$ and $F \neq \emptyset$ imply that $f(E) \neq \emptyset$ and $f(F) \neq \emptyset$, and this means that $f(E) \delta^{'} f(F)$. Therefore, $f$ is a pc-map. We now claim that a nonempty indiscrete proximity space is proximally contractible. Let $(X^{'},\delta^{'})$ be an indiscrete proximity space. Consider the pc-map $H : X^{'} \times I \rightarrow X^{'}$ defined by $H(x,0) = x_{0}$ and $H(x,t) = x$ for a fixed point $x_{0} \in X$ and $t \in (0,1]$. Then $H$ is a proximal homotopy between the constant map and the identity map on $X^{'}$. This leads to the fact that $(X^{'},\delta^{'})$ is proximally contractible. Finally, by Corollary \ref{cor1}, TC$(X^{'},\delta^{'})$ equals $1$ when $\delta^{'}$ is an indiscrete proximity on $X^{'}$.
\end{example}

\begin{figure}[h]
	\centering
	\includegraphics[width=1.00\textwidth]{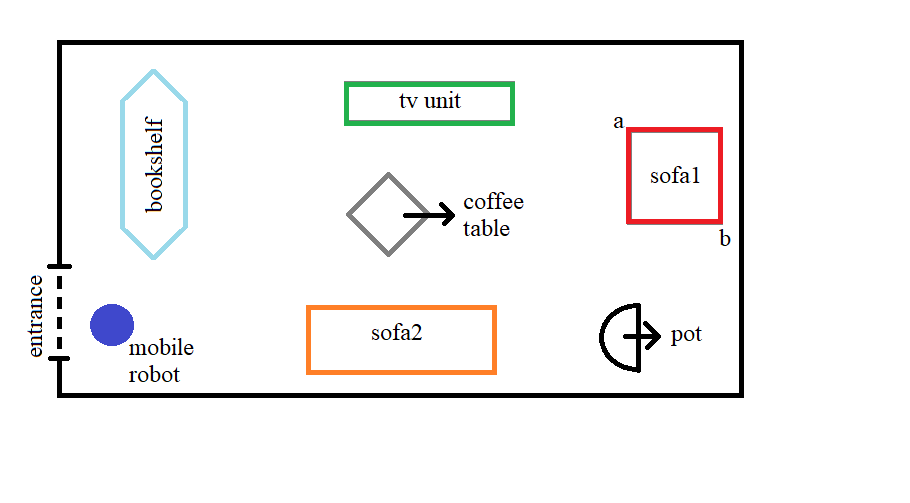}
	\caption{The projection of the lounge of a house on the floor.}
	\label{fig:1}
\end{figure}
\begin{example}\label{ex1}
	Consider the projection of the lounge of a house on the floor as in Figure \ref{fig:1}. First, the task is to compute the level of complexity of the motion required for the mobile robot to clean only around sofa1 in proximity spaces (see Figure \ref{fig:2} a)). The proximal complexity is not $1$ because the sofa1 is not proximal contractible. Indeed, when we consider two proximal paths $\alpha_{1}$ and $\alpha_{2}$ from $a$ to $b$ in sofa1 as the blue and green route, respectively, it is clear that $\alpha_{1}$ is far from (not near) $\alpha_{2}$. Therefore, the proximal motion planning algorithm from $\text{sofa1} \times \text{sofa1}$ to $P\text{sofa1}$ cannot be proximally continuous. This shows that the complexity must be bigger than $1$. Set sofa1 $= W_{1} \cup W_{2}$, where $W_{1}$ and $W_{2}$ are parts of sofa1 from $a$ to $b$ via $d$ and $c$, respectively. Then there exists a pc-map $s_{i} : W_{i} \rightarrow P\text{sofa1}$ for $i = 1,2$ that satisfies $\pi \circ s_{i} = 1_{W_{i}}$. This means that the proximal TC of sofa1, is $2$. Moreover, the result is the same for the bookshelf, the coffee table, the sofa2, the pot, or the tv unit.  
	\begin{figure}[h]
		\centering
		\includegraphics[width=0.40\textwidth]{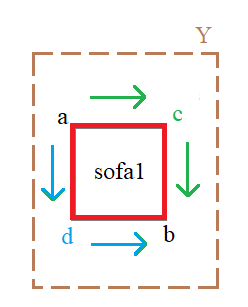}
		\caption{The motion planning for cleaning around the sofa1.}
		\label{fig:2}
	\end{figure}
\end{example}

\begin{proposition}\label{prop1}
	Let $\delta_{1}$ and $\delta_{2}$ be two proximities on $X$ such that $\delta_{1} > \delta_{2}$, i.e., $\delta_{1}$ is finer than $\delta_{2}$. Then for any two proximities $\delta_{1}^{'}$ and $\delta_{2}^{'}$ on $X \times X$ that is produced by $\delta_{1}$ and $\delta_{2}$, respectively, we have that $\delta_{1}^{'} > \delta_{2}^{'}$. 
\end{proposition}

\begin{proof}
	Let $(E_{1} \times E_{2}) \delta_{1}^{'} (F_{1} \times F_{2})$ on $X \times X$. Then we have that \[E_{1} \delta_{1} F_{1} \ \ \ \ \text{and} \ \ \ \ E_{2} \delta_{1} F_{2}.\] Since $\delta_{1}$ is finer than $\delta_{2}$, we get \[E_{1} \delta_{2} F_{1} \ \ \ \ \text{and} \ \ \ \ E_{2} \delta_{2} F_{2}.\] It follows that $(E_{1} \times E_{2}) \delta_{2}^{'} (F_{1} \times F_{2})$ on $X \times X$. Thus, we conclude that $\delta_{1}^{'} > \delta_{2}^{'}$. 
\end{proof}

\begin{theorem}\label{teo2}
	Let $\delta_{1}$ and $\delta_{2}$ be two proximities on $X$. Assume that $\delta_{1} > \delta_{2}$. Then TC$(X,\delta_{1}) \leq$ TC$(X,\delta_{2})$.
\end{theorem}

\begin{proof}
	Suppose that TC$(X,\delta_{1}) >$ TC$(X,\delta_{2})$. In other words, we have that TC$(X,\delta_{1}) = k$, TC$(X,\delta_{2}) = l$, and $k > l$. If TC$(X,\delta_{1}) = k$, then there exists a pc-map $s_{i} : U_{i} \subset X \times X \rightarrow PX$ for each $i \in \{1,\cdots,k\}$ with the property that $\pi_{1} \circ s_{i}$ is the identity map, where $\pi_{1} : (PX,\delta_{1}^{''}) \rightarrow (X \times X,\delta_{1}^{'})$ is the proximal path fibration. Similarly, if TC$(X,\delta_{2}) = l$, then there exists a pc-map $t_{j} : V_{j} \subset X \times X \rightarrow PX$ for each $j \in \{1,\cdots,l\}$ with the property that $\pi_{2} \circ t_{j}$ is the identity map, where $\pi_{2} : (PX,\delta_{2}^{''}) \rightarrow (X \times X,\delta_{2}^{'})$ is the proximal path fibration. For $l < m < k$, $s_{m}$ is a pc-map and satisfies that $(E_{1} \times E_{2}) \delta_{1}^{'} (F_{1} \times F_{2})$ implies $s_{m}(E_{1},E_{2}) \delta_{1}^{''} s_{m}(F_{1},F_{2})$. On the other hand, by Proposition \ref{prop1}, $(E_{1} \times E_{2}) \delta_{1}^{'} (F_{1} \times F_{2})$ implies that $(E_{1} \times E_{2}) \delta_{2}^{'} (F_{1} \times F_{2})$, which follows that $t_{m}(E_{1},E_{2}) \delta_{2}^{''} t_{m}(F_{1},F_{2})$. This is a contradiction because TC$(X,\delta_{2})$ cannot exceed $l$. As a consequence, TC$(X,\delta_{1}) \leq$ TC$(X,\delta_{2})$.
\end{proof}

\quad Let $(X,\tau)$ be a completely regular space. Then the subsets $E$ and $F$ of $X$ are said to be functionally distinguishable if and only if there exists a continuous function $f : X \rightarrow [0,1]$ with $f(E) = 0$ and $f(F) = 1$ \cite{NaimpallyWarrack:1970}. A proximity $\delta$ on a set $X$ can be defined by
\begin{eqnarray}\label{eq1}
	E \underline{\delta} F \ \ \text{if and only if} \ \ E \ \text{and} \ F \ \text{are functionally distinguishable}.
\end{eqnarray}

\begin{corollary}
	Let $(X,\tau)$ and $(X,\tau^{'})$ be two completely regular spaces such that $\tau \subset \tau^{'}$. Given any proximities $\delta$ and $\delta^{'}$ on $X$ with respect to $\tau$ and $\tau^{'}$, respectively, by considering (\ref{eq1}), we have that TC$(X,\delta) >$ TC$(X,\delta^{'})$.  
\end{corollary}

\begin{proof}
	The result is obtained by Theorem 2.17b of \cite{NaimpallyWarrack:1970} and Theorem \ref{teo2}.
\end{proof}

\begin{definition}\label{def1}
	Let $(X,\delta)$ be a proximity space and $Y \subset X \times X$. The relative proximal topological complexity TC$_{X}(Y,\delta_{Y})$ (or simply relative proximal complexity) is defined as the proximal Schwarz genus of the proximal fibration $\pi^{'} : PX^{Y} \rightarrow Y$, where $PX^{Y}$ is a subset of $PX$ that consists of all proximal paths $\alpha : [0,1] \rightarrow X$ with the property $(\alpha(0),\alpha(1)) \in Y$.
\end{definition}

\quad Note that, in Definition \ref{def1}, both $PX^{Y}$ and $Y$ have the respective induced (subspace) proximities. Then $\pi^{'}$ is a pc-map, which concludes that it is another example of a proximal fibration. The natural result of Definition \ref{def1} says that the equality TC$(X,\delta) =$ TC$_{X}(X \times X,\delta^{'})$ holds for a proximity $\delta^{'}$ defined on $X \times X$ if one considers $Y = X \times X$. Moreover, the relative proximal TC number is a lower bound for the proximal TC number, i.e.,
\begin{eqnarray*}
	\text{TC}_{X}(Y,\delta_{Y}) \leq \text{TC}(X,\delta).
\end{eqnarray*}

\begin{definition}\label{def2}
	Let $f : (X,\delta_{1}) \rightarrow (Y,\delta_{2})$ be a proximal fibration. Given a proximal fibration $\pi^{f} : (PX,\delta^{'}) \rightarrow (X \times Y,\delta^{''})$ with $\pi^{f} = (1 \times f) \circ \pi$, i.e., for any proximal path $\alpha$ on $X$, $\pi^{f}(\alpha) = (\alpha(0),f \circ \alpha(1))$, the proximal topological complexity of $f$, denoted by TC$_{\delta_{1},\delta_{2}}(f)$, is defined by genus$_{\delta^{'},\delta^{''}}(\pi^{f})$.
\end{definition}

\begin{example}
	\textbf{i)} Let $f = 1_{X} : (X,\delta) \rightarrow (X,\delta)$. Then $\pi^{1_{X}}$ corresponds to the proximal path fibration $\pi : PX \rightarrow X \times X$, that is,
	\begin{eqnarray*}
		\text{TC}_{\delta,\delta}(1_{X}) = \text{TC}(X,\delta).
	\end{eqnarray*}

    \textbf{ii)} Let $f : (X,\delta_{1}) \rightarrow (\{y_{0}\},\delta_{2})$ be the consant map. Then $f$ is a proximal fibration. For the proximal fibration $\pi^{f} : (PX,\delta^{'}) \rightarrow (X \times \{y_{0}\},\delta^{''})$, $\pi^{f}(\alpha) = (\alpha(0),y_{0})$, define a map
    \begin{eqnarray*}
    	&&s : (X \times \{y_{0}\},\delta^{''}) \longrightarrow (PX,\delta^{'}) \\
    	&&\hspace*{1.0cm} (x,y_{0}) \longmapsto s(x,y_{0}) = \beta_{x},
    \end{eqnarray*}
    where $\beta_{x}$ is the constant map at $x$. Given $E \times \{y_{0}\}$, $F \times \{y_{0}\} \in 2^{X \times \{y_{0}\}}$, the fact $(E \times \{y_{0}\}) \delta^{''} (F \times \{y_{0}\})$ implies that $E \delta_{1} F$. Then the constant proximal path at any point of $E$ is near the constant proximal path at any point of $F$, which follows that $s$ is a pc-map. Moreover, we have that $\pi^{f} \circ s$ is the identity on $X \times \{x_{0}\}$. Finally, TC$_{\delta_{1},\delta_{2}}(f) = 1$.
    
    \textbf{iii)} Consider any two proximity spaces $(X,\delta_{1})$ and $(Y,\delta_{2})$ such that the map $p_{2}$ from $(X \times Y,\delta)$ to $(Y,\delta_{2})$ is the second projection. For the proximal fibration $\pi^{p_{2}} : (P(X \times Y),\delta^{'}) \rightarrow ((X \times Y) \times Y,\delta^{''})$ with $\pi^{p_{2}}(\alpha) = (\alpha(0),p_{2}(\alpha(1)))$, define a map
    \begin{eqnarray*}
    	s : ((X \times Y) \times Y,\delta^{''}) \rightarrow (P(X \times Y),\delta^{'})
    \end{eqnarray*}
    given by $s((x_{1},y_{1}),y_{2})$ is a proximal path in $X \times Y$ from $(x_{1},y_{1})$ to $(x_{0},y_{2})$, where $x_{0}$ is a fixed point in $X$. Assume that $(E_{1} \times E_{2}) \times E_{3}$ is near $(F_{1} \times F_{2}) \times F_{3}$ for $(E_{1} \times E_{2}) \times E_{3}$, $(F_{1} \times F_{2}) \times F_{3} \in 2^{(X \times Y) \times Y}$. Then $E_{1} \times E_{2}$ is near $F_{1} \times F_{2}$ in $X \times Y$ and $E_{3}$ is near $F_{3}$ in $Y$. The second statement also implies that $\{x_{0}\} \times E_{3}$ is near $\{x_{0}\} \times F_{3}$ in $X \times Y$. Thus, the path from any point of $E_{1} \times E_{2}$ to any point of $\{x_{0}\} \times E_{3}$ is near the path from any point of $F_{1} \times F_{2}$ to any point of $\{x_{0}\} \times F_{3}$. This means that $s$ is a pc-map. Since $\pi_{1}^{p_{2}} \circ s$ is the identity map on $(X \times Y) \times Y$, we conclude that TC$_{\delta_{1},\delta_{2}}(p_{2}) = 1$.
\end{example} 

\begin{proposition}
	Let $(X,\delta)$, $(X,\delta_{1})$, $(X,\delta_{1}^{'})$ and $(Y,\delta^{'})$, $(Y,\delta_{2})$, $(Y,\delta_{2}^{'})$ be proximity spaces. Then the following hold:
	
	\textbf{i)} If $f : (X,\delta_{1}) \rightarrow (Y,\delta^{'})$ is a pc-map and $\delta_{1} > \delta_{1}^{'}$ on $X$, then TC$_{\delta_{1},\delta^{'}}(f)$ is greater than or equal to TC$_{\delta_{1}^{'},\delta^{'}}(f)$.
	
	\textbf{ii)} If $f : (X,\delta) \rightarrow (Y,\delta_{2}^{'})$ is a pc-map and $\delta_{2}^{'} > \delta_{2}$ on $Y$, then TC$_{\delta,\delta_{2}}(f)$ is less than or equal to TC$_{\delta,\delta_{2}^{'}}(f)$.
\end{proposition}

\begin{proof}
	\textbf{i)} Let TC$_{\delta_{1},\delta^{'}}(f) = m$. Then the proximal Schwarz genus of $\pi^{f}$ is equal to $m$. It follows that $X \times Y$ can be written as the union of $V_{1},\cdots,V_{m}$ and $\pi^{f}$ admits a pc-map $s_{i} : V_{i} \rightarrow PX$ such that $\pi^{f} \circ s_{i} = 1_{V_{i}}$ for each $i = 1,\cdots,m$. Assume that $\tau$ denotes a proximity on $X \times Y$. The proximal continuity of $s_{i}$ gives us that $s_{i}(E_{1},E_{2})(t) \delta_{1} s_{i}(F_{1},F_{2})(t)$ whenever $(E_{1} \times E_{2}) \tau (F_{1} \times F_{2})$ for $E_{1} \times E_{2}$, $F_{1} \times F_{2} \in 2^{V_{i}}$. Since $\delta_{1} > \delta_{1}^{'}$, we have that $s_{i}(E_{1},E_{2})(t) \delta_{1}^{'} s_{i}(F_{1},F_{2})(t)$ whenever $(E_{1} \times E_{2}) \tau (F_{1} \times F_{2})$. As a consequence, TC$_{\delta_{1}^{'},\delta^{'}}(f) \leq m$.
	
	\textbf{ii)} Let TC$_{\delta,\delta_{2}^{'}}(f) = m$. Then the proximal Schwarz genus of $\pi^{f}$ is equal to $m$. Therefore, we have that $X \times Y = V_{1} \cup \cdots \cup V_{m}$ and there exits a pc-map $s_{i} : V_{i} \rightarrow PX$ with the property $\pi^{f} \circ s_{i} = 1_{V_{i}}$ for each $i = 1,\cdots,m$. Assume that $X \times Y$ has a nearness relation $\tau^{'}$. The proximal continuity of $s_{i}$ says that $s_{i}(E_{1},E_{2})(t) \delta_{2}^{'} s_{i}(F_{1},F_{2})(t)$ whenever $(E_{1} \times E_{2}) \tau^{'} (F_{1} \times F_{2})$ for $E_{1} \times E_{2}$, $F_{1} \times F_{2} \in 2^{V_{i}}$. Since $\delta_{2}^{'} > \delta_{2}$, we have that $s_{i}(E_{1},E_{2})(t) \delta_{2} s_{i}(F_{1},F_{2})(t)$ whenever $(E_{1} \times E_{2}) \tau^{'} (F_{1} \times F_{2})$. As a consequence, TC$_{\delta,\delta_{2}}(f) \leq m$.
\end{proof}

\section{Proximal Homotopic Distances and Proximal Motion Planning Problem}
\label{sec:2}

\quad In this section, our main task is to show that the proximal TC is a special case of the introduced notion proximal homotopic distance. Also, we point out useful properties related to the homotopic distance.

\begin{definition}
	Let $f$ and $g$ be two pc-maps from $(X,\delta_{1})$ to $(Y,\delta_{2})$. The proximal homotopic distance (shortly called proximal distance) between $f$ and $g$, denoted by D$_{\delta_{1},\delta_{2}}(f,g)$, is the minimum integer $m>0$ if the following hold:
	\begin{itemize}
		\item $X$ has at least one cover $\{V_{1}, \cdots, V_{m}\}$.
		\item For all $j = 1, \cdots, m$, $f|_{V_{j}} \simeq_{\delta_{1},\delta_{2}} g|_{V_{j}}$.
	\end{itemize}
    If such a covering does not exist, then D$_{\delta_{1},\delta_{2}}(f,g) = \infty$.
\end{definition} 

\begin{proposition}\label{p4}
	Let $f$, $g : (X,\delta_{1}) \rightarrow (Y,\delta_{2})$ be pc-maps. Then
	
	\textbf{i)} D$_{\delta_{1},\delta_{2}}(f,g) =$ D$_{\delta_{1},\delta_{2}}(g,f)$.
	
	\textbf{ii)} D$_{\delta_{1},\delta_{2}}(f,g) = 1$ if and only if $f$ and $g$ are proximally homotopic.
	
	\textbf{iii)} If $X$ is finite and proximally connected, then D$_{\delta_{1},\delta_{2}}(f,g)$ is finite.
\end{proposition}

\begin{proof}
	\textbf{i)} It is clear from the fact that the statement $f|_{V_{j}} \simeq_{\delta_{1},\delta_{2}} g|_{V_{j}}$
    can be thought as $g|_{V_{j}} \simeq_{\delta_{1},\delta_{2}} f|_{V_{j}}$
    for all $j = 1, \cdots, m$, where $\{V_{1},\cdots,V_{m}\}$ is a covering of $X$.
    
    \textbf{ii)} Let D$_{\delta_{1},\delta_{2}}(f,g) = 1$. Then $X = V_{1}$ and we find that $f$ is proximally homotopic to $g$. Conversely, if $f$ and $g$ are proximally homotopic maps, then the restrictions $f|_{V_{1}}$ and $g|_{V_{1}}$ are proximally homotopic, where $X$ has a one-element covering $\{V_{1}\}$. Thus, we conclude that D$_{\delta_{1},\delta_{2}}(f,g) = 1$. 
    
    \textbf{iii)} Let $X$ be a finite proximity space. Then $\{\{x\} : x \in X\}$ is a finite covering of $X$. There exists a pc-map
    $h : [0,1] \rightarrow X$ with $h(0) = x$ and $h(1) = f^{-1}(g(x))$ because $X$ is proximally connected. Define a map
    \begin{eqnarray*}
    	&&G : U \times [0,1] \rightarrow Y \\
    	&&\hspace*{0.8cm} (x,t) \mapsto G(x,t) = f(h(t)).
    \end{eqnarray*}
    $G$ is a pc-map since $f$ and $h$ are pc-maps. Moreover, we have that
    \begin{eqnarray*}
    	&&G(x,0) = f(h(0)) = f(x), \\
    	&&G(x,1) = f(h(1)) = f(f^{-1}(g(x))) = g(x).
    \end{eqnarray*}
    It follows that $G$ is a proximal homotopy between $f|_{\{x\}}$ and $g|_{\{x\}}$. This shows that the proximal distance between $f$ and $g$ is finite.    
\end{proof}

\begin{proposition}\label{p2}
	Given pc-maps $f_{1}$, $f_{2}$, $g_{1}$, $g_{2} : (X,\delta_{1}) \rightarrow (Y,\delta_{2})$ such that $f_{1} \simeq_{\delta_{1},\delta_{2}} f_{2}$ and $g_{1} \simeq_{\delta_{1},\delta_{2}} g_{2}$, we have that D$_{\delta_{1},\delta_{2}}(f_{1},g_{1}) =$ D$_{\delta_{1},\delta_{2}}(f_{2},g_{2})$.
\end{proposition}

\begin{proof}
	We first show that D$_{\delta_{1},\delta_{2}}(f_{2},g_{2}) \leq$ D$_{\delta_{1},\delta_{2}}(f_{1},g_{1})$. Let D$_{\delta_{1},\delta_{2}}(f_{1},g_{1}) = m$. Then $X$ has a covering $\{V_{1},\cdots,V_{m}\}$ and for all $j = 1,\cdots,m$, we have that \[f_{1}|_{V_{j}} \simeq_{\delta_{1},\delta_{2}} g_{1}|_{V_{j}}.\] For all $j$, $f_{1}|_{V_{j}} \simeq_{\delta_{1},\delta_{2}} f_{2}|_{V_{j}}$ since $f_{1} \simeq_{\delta_{1},\delta_{2}} f_{2}$. Similarly, $g_{1} \simeq_{\delta_{1},\delta_{2}} g_{2}$ implies that $g_{1}|_{V_{j}} \simeq_{\delta_{1},\delta_{2}} g_{2}|_{V_{j}}$ for all $j$. By Theorem 11 of \cite{PetersTane2:2021}, we get
	\begin{eqnarray*}
		f_{2}|_{V_{j}} \simeq_{\delta_{1},\delta_{2}} g_{2}|_{V_{j}}
	\end{eqnarray*}
    for all $j$. This means that D$_{\delta_{1},\delta_{2}}(f_{2},g_{2}) \leq m$. A similar way can be used to show that D$_{\delta_{1},\delta_{2}}(f_{1},g_{1}) \leq$ D$_{\delta_{1},\delta_{2}}(f_{2},g_{2})$.
\end{proof}

\begin{lemma}\label{l1}
	Let $(X,\delta)$, $(X,\delta_{1})$, $(X,\delta_{1}^{'})$ and $(Y,\delta^{'})$, $(Y,\delta_{2})$, $(Y,\delta_{2}^{'})$ be proximity spaces. Then the following hold:
	
	\textbf{i)} If $f : (X,\delta) \rightarrow (Y,\delta_{2}^{'})$ is a pc-map and $\delta_{2}^{'} > \delta_{2}$ on $Y$, then $f : (X,\delta) \rightarrow (Y,\delta_{2})$ is a pc-map.
	
	\textbf{ii)} If $f : (X,\delta_{1}) \rightarrow (Y,\delta^{'})$ is a pc-map and $\delta_{1}^{'} > \delta_{1}$ on $X$, then $f : (X,\delta_{1}^{'}) \rightarrow (Y,\delta^{'})$ is a pc-map.
\end{lemma}

\begin{proof}
	\textbf{i)} We shall show that $E \delta F$ implies that $f(E) \delta_{2} f(F)$ for $E$, $F \in 2^{X}$. Let $E \delta F$. Then $f(E) \delta_{2}^{'} f(F)$ because $f : (X,\delta) \rightarrow (Y,\delta_{2}^{'})$ is a pc-map. By the assumption $\delta_{2}^{'} > \delta_{2}$, we get $f(E) \delta_{2} f(F)$. This gives the desired result.
	
	\textbf{ii)} The method is similar to the first part. Assume that $E \delta_{1}^{'} F$ for $E$, $F \in 2^{X}$. Since $\delta_{1}^{'} > \delta_{1}$, we have that $E \delta_{1} F$. By the proximal continuity of $f : (X,\delta_{1}) \rightarrow (Y,\delta^{'})$, we find that $f(E) \delta^{'} f(F)$, which completes the proof.
\end{proof}

\begin{theorem}\label{t1}
	Let $f_{1}$, $f_{2} : (X,\delta_{1}) \rightarrow (Y,\delta_{2})$ and $g_{1}$, $g_{2} : (X,\delta_{1}) \rightarrow (Y,\delta_{2}^{'})$ be pc-maps such that $f_{1} \simeq_{\delta_{1},\delta_{2}} f_{2}$ and $g_{1} \simeq_{\delta_{1},\delta_{2}^{'}} g_{2}$. If $\delta_{2}^{'} > \delta_{2}$, then D$_{\delta_{1},\delta_{2}}(f_{1},g_{1})$ is less than or equal to D$_{\delta_{1},\delta_{2}^{'}}(f_{2},g_{2})$.
\end{theorem}

\begin{proof}
	Let D$_{\delta_{1},\delta_{2}^{'}}(f_{2},g_{2}) = m$. Then $X$ has a covering $\{V_{1},\cdots,V_{m}\}$ and \[f_{2}|_{V_{j}} \simeq_{\delta_{1},\delta_{2}^{'}} g_{2}|_{V_{j}}\] for all $j = 1,\cdots,n$. By Lemma \ref{l1} i), we have that 
	\[f_{2}|_{V_{j}} \simeq_{\delta_{1},\delta_{2}} g_{2}|_{V_{j}}.\] 
	$f_{1} \simeq_{\delta_{1},\delta_{2}} f_{2}$ and $g_{1} \simeq_{\delta_{1},\delta_{2}^{'}} g_{2}$ imply that \[f_{1}|_{V_{j}} \simeq_{\delta_{1},\delta_{2}} f_{2}|_{V_{j}} \ \ \text{and} \ \ g_{1}|_{V_{j}} \simeq_{\delta_{1},\delta_{2}^{'}} g_{2}|_{V_{j}},\] respectively. Using Lemma \ref{l1} i) again, we find that
    \[g_{1}|_{V_{j}} \simeq_{\delta_{1},\delta_{2}} g_{2}|_{V_{j}}.\] 
	Therefore, by Theorem 11 of \cite{PetersTane2:2021}, we get
	\[f_{1}|_{V_{j}} \simeq_{\delta_{1},\delta_{2}} g_{1}|_{V_{j}},\]
	which means that D$_{\delta_{1},\delta_{2}}(f_{1},g_{1}) \leq m$.
\end{proof}

\begin{theorem}\label{t3}
	Let $f_{1}$, $f_{2} : (X,\delta_{1}) \rightarrow (Y,\delta^{'})$ and $g_{1}$, $g_{2} : (X,\delta_{1}^{'}) \rightarrow (Y,\delta^{'})$ be pc-maps such that $f_{1} \simeq_{\delta_{1},\delta^{'}} f_{2}$ and $g_{1} \simeq_{\delta_{1}^{'},\delta^{'}} g_{2}$. If $\delta_{1}^{'} > \delta_{1}$, then D$_{\delta_{1}^{'},\delta^{'}}(f_{2},g_{2})$ is less than or equal to D$_{\delta_{1},\delta^{'}}(f_{1},g_{1})$.
\end{theorem}

\begin{proof}
	The proof uses a similar method to Theorem \ref{t1}. Let D$_{\delta_{1},\delta^{'}}(f_{1},g_{1}) = m$. Then $X$ has a covering $\{V_{1},\cdots,V_{m}\}$ and \[f_{1}|_{V_{j}} \simeq_{\delta_{1},\delta^{'}} g_{1}|_{V_{j}}\] for all $j = 1,\cdots,n$. By Lemma \ref{l1} ii), we have that 
	\[f_{1}|_{V_{j}} \simeq_{\delta_{1}^{'},\delta^{'}} g_{1}|_{V_{j}}.\] 
	$f_{1} \simeq_{\delta_{1},\delta^{'}} f_{2}$ and $g_{1} \simeq_{\delta_{1}^{'},\delta^{'}} g_{2}$ imply that \[f_{1}|_{V_{j}} \simeq_{\delta_{1},\delta^{'}} f_{2}|_{V_{j}} \ \ \text{and} \ \ g_{1}|_{V_{j}} \simeq_{\delta_{1}^{'},\delta^{'}} g_{2}|_{V_{j}},\] respectively. Using by Lemma \ref{l1} ii) again, we find that
	\[f_{1}|_{V_{j}} \simeq_{\delta_{1}^{'},\delta^{'}} f_{2}|_{V_{j}}.\] 
	Therefore, by Theorem 11 of \cite{PetersTane2:2021}, we get
	\[f_{2}|_{V_{j}} \simeq_{\delta_{1}^{'},\delta^{'}} g_{2}|_{V_{j}},\]
	which means that D$_{\delta_{1}^{'},\delta^{'}}(f_{2},g_{2}) \leq m$.
\end{proof}

\quad Recall that TC of any proximity space is stated by the proximal Schwarz genus of a proximal fibration. The fact that there is a close relation between genus and D(f,g) in proximity spaces means that the proximal topological complexity can also be expressed over proximal homotopic distance. The following theorem clearly reveals the suggested relation.

\begin{theorem}\label{t2}
	Given two pc-maps $f$, $g : (X,\delta_{1}) \rightarrow (Y,\delta_{2})$ such that the diagram 
	\begin{displaymath}
		\xymatrix{
			E \ar[r]^{\pi_{2}} \ar[d]_{\pi_{1}} &
			PY \ar[d]^{\pi} \\
			X \ar[r]_{(f,g)} & Y \times Y}
	\end{displaymath}
	commutes, where $\pi_{1} : (E,\delta_{3}) \rightarrow (X,\delta_{1})$ is the pullback of the proximal path fibration on $Y$ by $(f,g) : (X,\delta_{1}) \rightarrow (Y \times Y,\delta^{'}_{2})$. Then 
	\begin{eqnarray*}
		\text{D}_{\delta_{1},\delta_{2}}(f,g) = \text{genus}_{\delta_{3},\delta_{1}}(\pi_{1}).
	\end{eqnarray*}
\end{theorem}

\begin{proof}
	Let $E = \{(x,\alpha) \subset X \times PY : \alpha(0) = f(x), \ \alpha(1) = g(x)\}$ and assume that genus$_{\delta_{3},\delta_{1}}(\pi_{1}) = m$. Then $X = U_{1} \cup \cdots \cup U_{m}$ and for each $i = 1,\cdots,m$, $s_{i} : U_{i} \rightarrow E$, $s_{i}(x) = (x,\alpha)$ is a pc-map such that $\pi_{1} \circ s = 1_{U_{i}}$. Define a proximal homotopy $G : U_{i} \times I \rightarrow Y \times Y$ with $G(x,t) = \alpha(t)$. Thus, we have that $f|_{U_{i}} \simeq_{\delta_{1},\delta_{2}} g|_{U_{i}}$ for each $i$. This means that D$_{\delta_{1},\delta_{2}}(f,g) \leq m$. On the other hand, if D$_{\delta_{1},\delta_{2}}(f,g) = m$, then $X$ has a cover $\{U_{1},\cdots,U_{m}\}$, and for all $i$, there is a proximal homotopy $F : U_{i} \times I \rightarrow Y$ such that $F(x,0) = f|_{U_{i}}$ and $F(x,1) = g|_{U_{i}}$. Define a pc-map $s_{i} : U_{i} \rightarrow E$ by $s_{i}(x) = (x,F_{x}(-))$, for each $i$. Therefore, we get \[\pi_{1} \circ s_{i}(x) = \pi_{1}(x,F_{x}(-)) = x = 1_{U_{i}}(x).\] This shows that genus$_{\delta_{3},\delta_{1}}(\pi_{1}) \leq m$.
\end{proof}

\begin{corollary}\label{c2}
	The proximal topological complexity of $(X,\delta)$ is D$_{\delta^{'},\delta}(p_{1},p_{2})$ for the projection maps $p_{1}, p_{2} : (X \times X,\delta^{'}) \rightarrow (X,\delta)$.
\end{corollary}

\begin{proof}
	Let $f = p_{1}$ and $g = p_{2}$ in Theorem \ref{t2}. This means that $(f,g) = 1_{X \times X}$. Thus, we conclude that TC$(X,\delta) =$ D$_{\delta^{'},\delta}(p_{1},p_{2})$.
\end{proof}

\begin{corollary}\label{c1}
	For two pc-maps $f, g : (X,\delta_{1}) \rightarrow (Y,\delta_{2})$, we have that D$_{\delta_{1},\delta_{2}}(f,g)$ is less than or equal to TC$(Y,\delta_{2})$.
\end{corollary}

\begin{proof}
	The proximal Schwarz genus of $\pi : PY \rightarrow Y \times Y$ is TC$(Y,\delta_{2})$ and by Theorem \ref{t2} the proximal Schwarz genus of $\pi_{1} : E \rightarrow X$ is D$_{\delta_{1},\delta_{2}}(f,g)$. Moreover, $\pi_{1}$ is the pullback of $\pi$. This shows that D$_{\delta_{1},\delta_{2}}(f,g) \leq$ TC$(Y,\delta_{2})$.
\end{proof}

\begin{proposition}\label{p1}
	Let $f, g : (X,\delta_{1}) \rightarrow (Y,\delta_{2})$ be pc-maps.
	
	\textbf{i)} If $h : (Y,\delta_{2}) \rightarrow (Z,\delta_{3})$ is a pc-map, then
	\begin{eqnarray*}
		\text{D}_{\delta_{1},\delta_{3}}(h \circ f,h \circ g) \leq \text{D}_{\delta_{1},\delta_{2}}(f,g). 
	\end{eqnarray*}
	\textbf{ii)} If $k : (Z,\delta_{3}) \rightarrow (X,\delta_{1})$ is a pc-map, then
	\begin{eqnarray*}
		\text{D}_{\delta_{3},\delta_{2}}(f \circ k,g \circ k) \leq \text{D}_{\delta_{1},\delta_{2}}(f,g). 
	\end{eqnarray*}
\end{proposition}

\begin{proof}
	\textbf{i)} Let D$_{\delta_{1},\delta_{2}}(f,g) = m$. Then $X$ can be covered by $m$ subsets $V_{1},\cdots V_{m}$, and for all $j = 1,\cdots,n$, we have that $f|_{V_{j}} \simeq_{\delta_{1},\delta_{2}} g|_{V_{j}}$. Since the composition of two pc-maps is a pc-map, we get
	\begin{eqnarray*}
		(h \circ f)|_{V_{j}} \simeq_{\delta_{1},\delta_{3}} h \circ f|_{V_{j}} \simeq_{\delta_{1},\delta_{3}} h \circ g|_{V_{j}} \simeq_{\delta_{1},\delta_{3}} (h \circ g)|_{V_{j}}.
	\end{eqnarray*}
    This proves that D$_{\delta_{1},\delta_{3}}(h \circ f,h \circ g) \leq m$.
    
    \textbf{ii)} Let D$_{\delta_{1},\delta_{2}}(f,g) = m$. Then $X$ can be written as the union of $m$ subsets $V_{1},\cdots V_{m}$ for which $f|_{V_{j}} \simeq_{\delta_{1},\delta_{2}} g|_{V_{j}}$ for all $j = 1,\cdots,n$. Therefore, $Z$ can be written as the union of $\{k^{-1}(V_{1}),\cdots,k^{-1}(V_{m})\}$ and the restriction $k_{j} : k^{-1}(V_{j}) \rightarrow Z$ is the composition $i \circ k$, where $i : V_{j} \rightarrow X$ is the inclusion map. Since the composition of two pc-maps is a pc-map, we obtain 
    \begin{eqnarray*}
    	(f \circ k)|_{k^{-1}(V_{j})} &\simeq_{\delta_{3},\delta_{2}}& f|_{k^{-1}(V_{j})} \circ k_{j} \simeq_{\delta_{3},\delta_{2}} g|_{k^{-1}(V_{j})} \circ k_{j} \\ &\simeq_{\delta_{3},\delta_{2}}& g|_{k^{-1}(V_{j})} \circ (i \circ k)|_{k^{-1}(V_{j})} \simeq_{\delta_{3},\delta_{2}} (g \circ (i \circ k))|_{k^{-1}(V_{j})} \\
    	&\simeq_{\delta_{3},\delta_{2}}& (g \circ k)|_{k^{-1}(V_{j})}.
    \end{eqnarray*}
    This proves that D$_{\delta_{3},\delta_{2}}(f \circ k,g \circ k) \leq m$.
\end{proof}

\begin{proposition}\label{p3}
	Let $f, g : (X,\delta_{1}) \rightarrow (Y,\delta_{2})$ be two pc-maps.
	
	\textbf{i)} If $h : (Y,\delta_{2}) \rightarrow (Z,\delta_{3})$ admits a left proximal homotopy inverse, then \[\text{D}_{\delta_{1},\delta_{3}}(h \circ f,h \circ g) = \text{D}_{\delta_{1},\delta_{2}}(f,g).\]
	
	\textbf{ii)} If $k : (Z,\delta_{3}) \rightarrow (X,\delta_{1})$ admits a right proximal homotopy inverse, then \[\text{D}_{\delta_{3},\delta_{2}}(f \circ k,g \circ k) = \text{D}_{\delta_{1},\delta_{2}}(f,g).\]
\end{proposition}

\begin{proof}
	\textbf{i)} Let $h^{'} : (Z,\delta_{3}) \rightarrow (Y,\delta_{2})$ be a pc-map with the property $h^{'} \circ h \simeq_{\delta_{2},\delta_{2}} 1_{Y}$. Then we have
	\begin{eqnarray*}
		\text{D}_{\delta_{1},\delta_{2}}(f,g) \geq \text{D}_{\delta_{1},\delta_{3}}(h \circ f,h \circ g) \geq \text{D}_{\delta_{1},\delta_{2}}(h^{'} \circ h \circ f,h^{'} \circ h \circ g) = \text{D}_{\delta_{1},\delta_{2}}(f,g)
	\end{eqnarray*}
	by Proposition \ref{p1} i) and Proposition \ref{p2}.
	
	\textbf{ii)} Let $k^{'} : (X,\delta_{1}) \rightarrow (Z,\delta_{3})$ be a pc-map with the property $k \circ k^{'} \simeq_{\delta_{1},\delta_{1}} 1_{X}$. Then we get
	\begin{eqnarray*}
		\text{D}_{\delta_{1},\delta_{2}}(f,g) \geq \text{D}_{\delta_{3},\delta_{2}}(f \circ k,g \circ k) \geq \text{D}_{\delta_{1},\delta_{2}}(f \circ k \circ k^{'},g \circ k \circ k^{'}) = \text{D}_{\delta_{1},\delta_{2}}(f,g)
	\end{eqnarray*}
	by Proposition \ref{p1} i) and Proposition \ref{p2}.
\end{proof}

\begin{theorem}\label{t4}
	Let $f, g : (X,\delta_{1}) \rightarrow (Y,\delta_{2})$ and $f^{'}, g^{'} : (X^{'},\delta_{1}^{'}) \rightarrow (Y^{'},\delta_{2}^{'})$ be any pc-maps. If $h : (X^{'},\delta_{1}^{'}) \rightarrow (X,\delta_{1})$ and $k : (Y,\delta_{2}) \rightarrow (Y^{'},\delta_{2}^{'})$ are proximal homotopy equivalences such that the diagram
	\begin{displaymath}
		\xymatrix{
			X \ar@<1ex>[r]_{g} \ar@<2ex>[r]^{f} &
			Y \ar[d]^{k} \\
			X^{'} \ar@<1ex>[r]_{g^{'}} \ar@<2ex>[r]^{f^{'}} \ar[u]_{h} & Y^{'}}
	\end{displaymath}
	commutes, then D$_{\delta_{1},\delta_{2}}(f,g) =$ D$_{\delta_{1}^{'},\delta_{2}^{'}}(f^{'},g^{'})$.
\end{theorem}

\begin{proof}
	Let $h$ and $k$ be proximal homotopy equivalences, i.e., they admit both left and right proximal homotopy inverses. From Proposition \ref{p2} and Proposition \ref{p3}, respectively, we obtain
	\begin{eqnarray*}
		\text{D}_{\delta_{1}^{'},\delta_{2}^{'}}(f^{'},g^{'}) = \text{D}_{\delta_{1}^{'},\delta_{2}^{'}}(k \circ f \circ h,k \circ g \circ h) = \text{D}_{\delta_{1}^{'},\delta_{2}}(f \circ h,g \circ h) = \text{D}_{\delta_{1},\delta_{2}}(f,g).
	\end{eqnarray*}
\end{proof}

\quad For ease of computation, it is important that TC$(X,\delta)$ is a proximal homotopy invariant. It is quite simple to derive this result from the previous theorem, which highly shortens the way of TC computations:

\begin{corollary}
	The proximal topological complexity is an invariant of proximal homotopy.
\end{corollary}

\begin{definition}
	Let $f, g : (X,\delta_{1}) \rightarrow (Y,\delta_{2})$ be two pc-maps with $A \subset X$. Then the relative proximal homotopic distance (shortly called relative proximal distance) between $f$ and $g$ on $A$, denoted by D$^{X}_{\delta_{1},\delta_{2}}(A;f,g)$, is D$_{\delta_{1},\delta_{2}}(f|_{A},g|_{A})$.
\end{definition}

\quad We observe that the relative proximal distance coincides with the proximal distance if $A = X$. For the inclusion $j : (A,\delta_{1}) \rightarrow (X,\delta_{1})$, the relative proximal distance is expressed as D$_{\delta_{1},\delta_{2}}(f \circ j, g \circ j)$. We now rewrite Definition \ref{def1} by using the relative proximal distance as follows.

\begin{proposition}
	For the inclusion map $j : (A,\delta_{1}) \rightarrow (X \times X,\delta^{'})$ and the projection map $p_{i} : (X \times X,\delta^{'}) \rightarrow (X,\delta_{1})$ with $i \in \{1,2\}$, the relative proximal complexity is given by
	\begin{eqnarray*}
		\text{TC}_{X}(A,\delta) = \text{D}_{\delta_{1},\delta_{1}}(p_{1} \circ j,p_{2} \circ j).
	\end{eqnarray*}
\end{proposition}

\quad It is a reasonable idea to extend the problem of finding the proximal TC number of a space to the problem of finding the TC number of fibrations that admits that space as a domain. This positively affects the course of the motion planning algorithm problem in proximity spaces. In other words, the computation of proximal TC$(f)$ provides the advantage of looking at the computation of proximal TC$(X)$ from a broad perspective.

\begin{definition}
	Given a proximal fibration $f : (X,\delta_{1}) \rightarrow (Y,\delta_{2})$, the projection maps $\pi_{1} : (X \times Y,\delta) \rightarrow (X,\delta_{1})$ and $\pi_{2} : (X \times Y,\delta) \rightarrow (Y,\delta_{2})$, the proximal topological complexity of $f$ (or simply proximal complexity of $f$) is defined as TC$_{\delta_{1},\delta_{2}}(f) =$ D$_{\delta,\delta_{2}}(f \circ \pi_{1},\pi_{2})$.
\end{definition}

\begin{lemma}\label{l2}
	Let $f : (X,\delta_{1}) \rightarrow (Y,\delta_{2})$ and $g : (Y,\delta_{2}) \rightarrow (Z,\delta_{3})$ be two proximal fibrations.
	
	\textbf{i)} If there exists a pc-map $f^{'} : (Y,\delta_{2}) \rightarrow (X,\delta_{1})$ such that $f \circ f^{'} \simeq_{\delta_{2},\delta_{2}} 1_{Y}$, then TC$_{\delta_{2},\delta_{3}}(g) \leq$ TC$_{\delta_{1},\delta_{3}}(g \circ f)$.
	
	\textbf{ii)} If there exists a pc-map $f^{'} : (Y,\delta_{2}) \rightarrow (X,\delta_{1})$ such that $f^{'} \circ f \simeq_{\delta_{1},\delta_{1}} 1_{X}$ and $f \circ f^{'} = 1_{Y}$, then TC$_{\delta_{1},\delta_{3}}(g \circ f) \leq$ TC$_{\delta_{2},\delta_{3}}(g)$.
\end{lemma}

\begin{proof}
	Consider the projection maps $\pi_{1}$, $\pi_{2}$, $\pi_{1}^{'}$, and $\pi_{2}^{'}$ from $(Y \times Z,\delta)$ to $(Y,\delta_{2})$, from $(Y \times Z,\delta)$ to $(Z,\delta_{3})$, from $(X \times Z,\delta^{'})$ to $(X,\delta_{1})$, and from $(X \times Z,\delta^{'})$ to $(Z,\delta_{3})$, respectively. 
	
	\textbf{i)} $f \circ f^{'} \simeq_{\delta_{2},\delta_{2}} 1_{Y}$ implies that $\pi_{1} \simeq_{\delta,\delta_{2}} f \circ \pi_{1}^{'} \circ (f^{'} \times 1_{Z})$ and $\pi_{2} \simeq_{\delta,\delta_{3}} \pi_{2}^{'} \circ (f^{'} \times 1_{Z})$. Since TC$_{\delta_{1},\delta_{3}}(g \circ f) =$ D$_{\delta^{'},\delta_{3}}(g \circ f \circ \pi_{1}^{'},\pi_{2}^{'})$ and TC$_{\delta_{2},\delta_{3}}(g) =$ D$_{\delta,\delta_{3}}(g \circ \pi_{1},\pi_{2})$, by Proposition \ref{p1}, we have that
	\begin{eqnarray*}
		\text{D}_{\delta,\delta_{3}}(g \circ \pi_{1},\pi_{2}) &=& \text{D}_{\delta,\delta_{3}}(g \circ f \circ \pi_{1}^{'} \circ (f^{'} \times 1_{Z}), \pi_{2}^{'} \circ (f^{'} \times 1_{Z}))\\
		&\leq& \text{D}_{\delta^{'},\delta_{3}}(g \circ f \circ \pi_{1}^{'},\pi_{2}^{'}).
	\end{eqnarray*}
    This proves that TC$_{\delta_{2},\delta_{3}}(g) \leq$ TC$_{\delta_{1},\delta_{3}}(g \circ f)$.
    
    \textbf{ii)} $f^{'} \circ f \simeq_{\delta_{1},\delta_{1}} 1_{X}$ implies that $\pi_{1}^{'}$ and $\pi_{2}^{'}$ homotopically as $f^{'} \circ \pi_{1} \circ (f \times 1_{Z})$ and $\pi_{2} \circ (f \times 1_{Z})$, respectively. Since TC$_{\delta_{1},\delta_{3}}(g \circ f) =$ D$_{\delta^{'},\delta_{3}}(g \circ f \circ \pi_{1}^{'},\pi_{2}^{'})$ and TC$_{\delta_{2},\delta_{3}}(g) =$ D$_{\delta,\delta_{3}}(g \circ \pi_{1},\pi_{2})$, by Proposition \ref{p2} and Proposition \ref{p1}, we have that
    \begin{eqnarray*}
    	\text{D}_{\delta^{'},\delta_{3}}(g \circ f \circ \pi_{1}^{'},\pi_{2}^{'}) &=&  \text{D}_{\delta^{'},\delta_{3}}(g \circ f \circ f^{'} \circ \pi_{1} \circ (f \times 1_{Z}),\pi_{2} \circ (f \times 1_{Z})) \\
    	&=& \text{D}_{\delta^{'},\delta_{3}}(g \circ \pi_{1} \circ (f \times 1_{Z}),\pi_{2} \circ (f \times 1_{Z})) \\
    	&\leq&\text{D}_{\delta,\delta_{3}}(g \circ \pi_{1},\pi_{2}).
    \end{eqnarray*}
    Consequently, we conclude that TC$_{\delta_{1},\delta_{2}}(g \circ f) \leq$ TC$_{\delta_{2},\delta_{3}}(g)$.
\end{proof}

\quad In proximity spaces, the fiber homotopy equivalence is defined as follows:

\begin{definition}
	Let $f : (X,\delta_{1}) \rightarrow (Y,\delta_{2})$ and $f^{'} : (X^{'},\delta_{1}^{'}) \rightarrow (Y,\delta_{2})$ be two proximal fibrations. If the following diagram is commutative with the property that $h : (X,\delta_{1}) \rightarrow (X^{'},\delta_{1}^{'})$ and $k : (X^{'},\delta_{1}^{'}) \rightarrow (X,\delta_{1})$ are proximal homotopy inverses for each other, namely that $h \circ k \simeq_{\delta_{1}^{'},\delta_{1}^{'}} 1_{X^{'}}$ and $k \circ h \simeq_{\delta_{1},\delta_{1}} 1_{X}$, then $f$ and $f^{'}$ are called proximal fiber homotopy equivalent fibrations.
	$$\xymatrix{
		X \ar[dr]_{f} \ar@<1ex>[rr]^h
		& & X^{'} \ar@<1ex>[ll]^k \ar[dl]^{f^{'}} \\
		& Y & }$$
\end{definition}

\begin{theorem}
	The proximal complexity of proximal fibrations is a proximal fiber homotopy equivalence invariant.
\end{theorem}

\begin{proof}
	Let $f : (X,\delta_{1}) \rightarrow (Y,\delta_{2})$ and $f^{'} : (X^{'},\delta_{1}^{'}) \rightarrow (Y,\delta_{2})$ be proximal fiber homotopy equivalent fibrations. Then we have that $f = f^{'} \circ h$ and $f^{'} = f \circ k$ for proximal homotopy inverses $h$ and $k$. By Lemma \ref{l2} i), we find that
	\begin{eqnarray*}
		\text{TC}_{\delta_{1},\delta_{2}}(f) \leq \text{TC}_{\delta_{1}^{'},\delta_{2}}(f \circ k) = \text{TC}_{\delta_{1}^{'},\delta_{2}}(f^{'}) \leq \text{TC}_{\delta_{1},\delta_{2}}(f^{'} \circ h) = \text{TC}_{\delta_{1},\delta_{2}}(f).
	\end{eqnarray*}
    Finally, we get TC$_{\delta_{1},\delta_{2}}(f) =$  TC$_{\delta_{1}^{'},\delta_{2}}(f^{'})$.
\end{proof}

\section{Higher Versions}
\label{sec:3}

\quad Studying generalized versions of concepts such as proximal TC and proximal D$(f,g)$ provides a different perspective on the proximal navigation problem. 

\begin{definition}
	Let $f_{1}, f_{2}, \cdots, f_{n}$ be pc-maps from $(X,\delta_{1})$ to $(Y,\delta_{2})$. The proximal $n$th (or higher) homotopic distance between $f_{1}, f_{2}, \cdots, f_{n}$, denoted by D$_{\delta_{1},\delta_{2}}(f_{1},f_{2},\cdots,f_{n})$, is the minimum integer $m>0$ if the following hold:
	\begin{itemize}
		\item $X$ has at least one cover $\{V_{1}, \cdots, V_{m}\}$.
		\item For all $j = 1, \cdots, m$, $f_{1}|_{V_{j}} \simeq_{\delta_{1},\delta_{2}} f_{2}|_{V_{j}} \simeq_{\delta_{1},\delta_{2}} \cdots \simeq_{\delta_{1},\delta_{2}} f_{n}|_{V_{j}}$.
	\end{itemize}
	If such a covering does not exist, then D$_{\delta_{1},\delta_{2}}(f_{1},f_{2},\cdots,f_{n})$ is $\infty$.
\end{definition} 

\begin{theorem}\label{t6}
	Let $f_{1},f_{2},\cdots,f_{n} : (X,\delta_{1}) \rightarrow (Y,\delta_{2})$ be any pc-maps. Then the following hold:
	
	\textbf{i)} Assume that $\sigma$ is any permutation of $\{1,2,\cdots,n\}$. Then
	\begin{eqnarray*}
		\text{D}_{\delta_{1},\delta_{2}}(f_{1},f_{2},\cdots,f_{n}) = \text{D}_{\delta_{1},\delta_{2}}(f_{\sigma(1)},f_{\sigma(2)},\cdots,f_{\sigma(n)}).
	\end{eqnarray*} 
	
	\textbf{ii)} D$_{\delta_{1},\delta_{2}}(f_{1},f_{2},\cdots,f_{n}) = 1$ if and only if $f_{1},f_{2},\cdots,f_{n}$ are proximally homotopic distance.
	
	\textbf{iii)} If $X$ is finite and proximally connected, then D$_{\delta_{1},\delta_{2}}(f_{1},f_{2},\cdots,f_{n})$ is finite.
	
	\textbf{iv)} Assume that $1 < r < n$. Then
	\begin{eqnarray*}
		\text{D}_{\delta_{1},\delta_{2}}(f_{1},f_{2},\cdots,f_{r}) \leq  \text{D}_{\delta_{1},\delta_{2}}(f_{1},f_{2},\cdots,f_{n}).
	\end{eqnarray*}
\end{theorem}

\begin{proof}
	See the method in the proof of Proposition \ref{p4} except for the last part. The last condition \textbf{iv)} is clearly obtained from the definition of the proximal higher homotopic distance.
\end{proof}
	
\begin{theorem}\label{t7}
		\textbf{i)} Consider any pc-maps $f_{1},f_{2},\cdots,f_{n}$ and $g_{1},g_{2},\cdots,g_{n}$ from $(X,\delta_{1})$ to $(Y,\delta_{2})$. Then  \[\text{D}_{\delta_{1},\delta_{2}}(f_{1},f_{2},\cdots,f_{n}) =  \text{D}_{\delta_{1},\delta_{2}}(g_{1},g_{2},\cdots,g_{n})\] if $f_{i} \simeq_{\delta_{1},\delta_{2}} g_{i}$ for each $i \in \{1,\cdots,n\}$.
		
		\textbf{ii)} Let $f_{1}$,$f_{2},\cdots,f_{n} : (X,\delta_{1}) \rightarrow (Y,\delta_{2})$ and $g_{1}$,$g_{2},\cdots,g_{n} : (X,\delta_{1}) \rightarrow (Y,\delta_{2}^{'})$ be pc-maps such that $f_{i} \simeq_{\delta_{1},\delta_{2}} f_{i+1}$ and $g_{i} \simeq_{\delta_{1},\delta_{2}^{'}} g_{i+1}$ for each $i \in \{1,\cdots,n\}$. If $\delta_{2}^{'} > \delta_{2}$, then \[\text{D}_{\delta_{1},\delta_{2}}(f_{1},f_{2},\cdots,f_{n}) \leq \text{D}_{\delta_{1},\delta_{2}^{'}}(g_{1},g_{2},\cdots,g_{n}).\]
		
		\textbf{iii)} Let $f_{1}$,$f_{2},\cdots,f_{n} : (X,\delta_{1}) \rightarrow (Y,\delta^{'})$ and $g_{1}$,$g_{2},\cdots,g_{n} : (X,\delta_{1}^{'}) \rightarrow (Y,\delta^{'})$ be pc-maps such that $f_{i} \simeq_{\delta_{1},\delta^{'}} f_{i+1}$ and $g_{i} \simeq_{\delta_{1}^{'},\delta^{'}} g_{i+1}$ for each $i \in \{1,\cdots,n\}$. If $\delta_{1}^{'} > \delta_{1}$, then \[\text{D}_{\delta_{1}^{'},\delta^{'}}(f_{1},f_{2},\cdots,f_{n}) \leq \text{D}_{\delta_{1},\delta^{'}}(g_{1},g_{2},\cdots,g_{n}).\]
\end{theorem}

\begin{proof}
	For the first part, see the method in the proof of Proposition \ref{p2}. Other parts follow the way that is similar to the proof of Theorem \ref{t1} and Theorem \ref{t3}, respectively.
\end{proof}
	
\begin{theorem}\label{t5}
	Let $f_{1},f_{2},\cdots,f_{n} : (X,\delta_{1}) \rightarrow (Y,\delta_{2})$ be pc-maps.
	
		\textbf{i)} If $h : (Y,\delta_{2}) \rightarrow (Z,\delta_{3})$ is a pc-map, then
		\begin{eqnarray*}
			\text{D}_{\delta_{1},\delta_{3}}(h \circ f_{1},h \circ f_{2},\cdots,h \circ f_{n}) \leq \text{D}_{\delta_{1},\delta_{2}}(f_{1},f_{2},\cdots f_{n}). 
		\end{eqnarray*}
		
		\textbf{ii)} If $k : (Z,\delta_{3}) \rightarrow (X,\delta_{1})$ is a pc-map, then
		\begin{eqnarray*}
			\text{D}_{\delta_{3},\delta_{2}}(f_{1} \circ k,f_{2} \circ k,\cdots,f_{n} \circ k) \leq \text{D}_{\delta_{1},\delta_{2}}(f_{1},f_{2},\cdots,f_{n}). 
		\end{eqnarray*}
		
		\textbf{iii)} If $h : (Y,\delta_{2}) \rightarrow (Z,\delta_{3})$ admits a left proximal homotopy inverse, then \[\text{D}_{\delta_{1},\delta_{3}}(h \circ f_{1},h \circ f_{2},\cdots,h \circ f_{n}) = \text{D}_{\delta_{1},\delta_{2}}(f_{1},f_{2},\cdots,f_{n}).\]
		
		\textbf{iv)} If $k : (Z,\delta_{3}) \rightarrow (X,\delta_{1})$ admits a right proximal homotopy inverse, then \[\text{D}_{\delta_{3},\delta_{2}}(f_{1} \circ k,f_{2} \circ k,\cdots,f_{n} \circ k) = \text{D}_{\delta_{1},\delta_{2}}(f_{1},f_{2},\cdots,f_{n}).\]
\end{theorem}

\begin{proof}
	See the method in the proof of Proposition \ref{p1} for the first two parts. For the others, see the method in the proof of Proposition \ref{p3}.
\end{proof}

\begin{definition}\label{def3}
	\textbf{i)} Let $f_{1},f_{2},\cdots,f_{n} : (X,\delta_{1}) \rightarrow (Y,\delta_{2})$ be pc-maps with $A \subset X$. Then the relative proximal higher homotopic distance between $f_{1},f_{2},\cdots,f_{n}$ on $A$, denoted by D$^{X}_{\delta_{1},\delta_{2}}(A;f_{1},f_{2},\cdots,f_{n})$, is D$_{\delta_{1},\delta_{2}}(f_{1}|_{A},f_{2}|_{A},\cdots,f_{n}|_{A})$.
	
	\textbf{ii)} Let $n>1$, $f = (f_{1},f_{2},\cdots,f_{n}) : (X,\delta_{1}) \rightarrow (Y^{n},\delta^{'})$ be a proximal fibration, and $\pi_{i} : (X^{n},\delta) \rightarrow (X,\delta_{1})$ be the $i$th projection map for each $i \in \{1,\cdots,n\}$. Then the proximal higher topological complexity of $f$ is defined as \[\text{TC}_{n,\delta_{1},\delta^{'}}(f) = \text{D}_{\delta,\delta^{'}}(f \circ \pi_{1},f \circ \pi_{2},\cdots,f \circ \pi_{n}).\]
\end{definition}

\begin{corollary}
	The proximal higher topological complexity of $(X,\delta)$, denoted by TC$_{n}(X,\delta)$, is D$_{\delta^{'},\delta}(p_{1},p_{2},\cdots,p_{n})$ for the $i$th projection map $p_{i}$ from $(X \times X,\delta^{'})$ to $(X,\delta)$ for each $i \in \{1,\cdots,n\}$.
\end{corollary}

\begin{proof}
	It is a general version of Corollary \ref{c2}.
\end{proof}

\begin{theorem}\label{t8}
	\textbf{i)} Given pc-maps $f_{1},f_{2},\cdots,f_{n} : (X,\delta_{1}) \rightarrow (Y,\delta_{2})$, we have that D$_{\delta_{1},\delta_{2}}(f_{1},f_{2},\cdots,f_{n}) \leq$ TC$_{n}(Y,\delta_{2})$.
	
	\textbf{ii)} Let $f_{1},f_{2},\cdots,f_{n} : (X,\delta_{1}) \rightarrow (Y,\delta_{2})$ and $f_{1}^{'},f_{2}^{'},\cdots,f_{n}^{'} : (X^{'},\delta_{1}^{'}) \rightarrow (Y^{'},\delta_{2}^{'})$ be pc-maps. If $h : (X^{'},\delta_{1}^{'}) \rightarrow (X,\delta_{1})$ and $k : (Y,\delta_{2}) \rightarrow (Y^{'},\delta_{2}^{'})$ are proximal homotopy equivalences such that the diagram
	$$\xymatrix{
		X \ar[r]^{f_{1},\cdots,f_{n}} &
		Y \ar[d]^{k} \\
		X^{'} \ar[u]^{h} \ar[r]_{f_{1}^{'},\cdots,f_{n}^{'}} & Y^{'}}$$
	commutes, then D$_{\delta_{1},\delta_{2}}(f_{1},f_{2},\cdots,f_{n}) =$ D$_{\delta_{1}^{'},\delta_{2}^{'}}(f_{1}^{'},f_{2}^{'},\cdots,f_{n^{'}})$.
	
	\textbf{iii)} The proximal higher topological complexity is an invariant of proximal homotopy.
	
	\textbf{iv)} Let $j : (A,\delta_{1}) \rightarrow (X^{n},\delta^{'})$ and $p_{i} : (X^{n},\delta^{'}) \rightarrow (X,\delta_{1})$ be the inclusion map and $i$th projection map, respectively, with $i \in \{1,\cdots,n\}$, the relative proximal higher complexity is given by
		\begin{eqnarray*}
			\text{TC}_{n,X}(A,\delta) = \text{D}_{\delta_{1},\delta_{1}}(p_{1} \circ j,p_{2} \circ j,\cdots,p_{n} \circ j).
		\end{eqnarray*}
	
	\textbf{v)} The proximal higher topological complexity of proximal fibrations is a proximal fiber homotopy equivalence invariant. 
\end{theorem}

\begin{proof}
	See the method in the proof of Corollary \ref{c1} for the first result. For the second, see the method in the proof of Theorem \ref{t4}. By \textbf{ii)}, we conclude that the proximal higher topological complexity is a proximal homotopy invariant. The last two parts are direct consequences of Definition \ref{def3}.
\end{proof}

\quad For any proximity space $(X,\delta_{1})$, TC$_{n}(X,\delta_{1})$ is always equal to $1$ and  also TC$_{2}(X,\delta_{1}) =$ TC$(X,\delta_{1})$. By Theorem \ref{t5}, we have that TC$_{n}(X,\delta_{1}) \geq$ TC$_{n}(Y,\delta_{2})$ when $X$ dominates $Y$, which confirms Theorem \ref{t8} iii). Furthermore, the natural consequence of Theorem \ref{t6} \textbf{iv)} is the inequality TC$_{n}(X,\delta) \leq$ TC$_{n+1}(X,\delta)$.

\begin{theorem}
	\textbf{i)} Let $f$, $g : (X,\delta_{1}) \rightarrow (Y^{n},\delta^{'})$ be two proximal fibrations such that $f \simeq_{\delta_{1},\delta^{'}} g$, i.e., $f$ and $g$ are proximal homotopic to each other. Then \[\text{TC}_{n,\delta_{1},\delta^{'}}(f) = \text{TC}_{n,\delta_{1},\delta^{'}}(g).\]
	
	\textbf{ii)} Let $f : (X,\delta_{1}) \rightarrow (Y^{n},\delta^{'})$ be a proximal fibration. Then \[\text{TC}_{n,\delta_{1},\delta^{'}}(f) \leq \text{TC}_{n}(X,\delta_{1}).\]
\end{theorem}

\begin{proof}
    \textbf{i)} Let $\pi_{i} : (X^{n},\delta) \rightarrow (X,\delta_{1})$ be the $i$th projection map for all $i \in \{1,\cdots,n\}$. Since $f \simeq_{\delta_{1},\delta^{'}} g$, we have that $f \circ \pi_{i} \simeq_{\delta,\delta^{'}} g \circ \pi_{i}$ for each $i$. By Theorem \ref{t7} \textbf{i)}, we get
    \begin{eqnarray*}
    	\text{D}_{\delta,\delta^{'}}(f \circ \pi_{1},f \circ \pi_{2},\cdots,f \circ \pi_{n}) = \text{D}_{\delta,\delta^{'}}(g \circ \pi_{1},g \circ \pi_{2},\cdots,g \circ \pi_{n}).
    \end{eqnarray*} 
    Consequently, we find that TC$_{n,\delta_{1},\delta^{'}}(f) =$ TC$_{n,\delta_{1},\delta^{'}}(g)$.
    
    \textbf{ii)} Given the $i$th projection $\pi_{i} : (X^{n},\delta) \rightarrow (X,\delta_{1})$ with $i \in \{1,\cdots,n\}$, by Theorem \ref{t5} \textbf{i)}, we have that
    \begin{eqnarray*}
    	\text{D}_{\delta,\delta^{'}}(f \circ \pi_{1},f \circ \pi_{2}\cdots,f \circ \pi_{n}) \leq \text{D}_{\delta,\delta_{1}}(\pi_{1},\pi_{2},\cdots,\pi_{n}),
    \end{eqnarray*}
    which means that TC$_{n,\delta_{1},\delta^{'}}(f) \leq$ TC$_{n}(X,\delta_{1})$.
\end{proof}

\section{TC Numbers of Descriptive Proximity Spaces}
\label{sec:4}

\quad Let $\alpha_{1} : [0,1] \rightarrow X$ and $\alpha_{2} : [0,1] \rightarrow X$ be two descriptive proximal paths in $X$. $\alpha_{1}$ and $\alpha_{2}$ are descriptively near if for any $E$, $F \in 2^{I}$, $E \delta_{\Phi}^{'} F$ implies that $\alpha_{1}(E) \delta_{\Phi} \alpha_{2}(F)$, where $\delta_{\Phi}^{'}$ is a descriptive proximity on $I$ and $\delta_{\Phi}$ is a descriptive proximity on $X$. Hence, $\pi : PX \rightarrow X \times X$ with $\pi(\alpha) = (\alpha(0),\alpha(1))$ is said to be a descriptive proximal path fibration, where $PX$ is given by the set \[\{\alpha \ | \ \alpha : [0,1] \rightarrow X \ \ \text{is a descriptive proximal path}\}.\]

\begin{definition}
	Let $p : (X,\delta_{\Phi}) \rightarrow (X^{'},\delta_{\Phi}^{'})$ be a descriptive proximal fibration. Then a descriptive proximal Schwarz genus of $p$ is the possible minimum integer $m > 0$ if $X^{'}$ has a cover $\{U_{1},U_{2},\cdots,U_{m}\}$, i.e., $X^{'}$ can be written as the union of subsets $U_{1}$, $U_{2}$, $\cdots$, $U_{m}$, such that there exists a dpc-map $s_{i} : U_{i} \rightarrow X$ with the property $p \circ s_{i} = 1_{U_{i}}$ for each $i \in \{1,\cdots,m\}$.
\end{definition}

The descriptive proximal Schwarz genus of $p : (X,\delta_{\Phi}) \rightarrow (X^{'},\delta_{\Phi}^{'})$ is denoted by genus$_{\delta_{\Phi},\delta_{\Phi}^{'}}(p)$.

\begin{definition}\label{def4}
	Let $(X,\delta_{\Phi})$ be a connected descriptive proximity space and the map $\pi : PX \rightarrow X \times X$, $\pi(\alpha) = (\alpha(0),\alpha(1))$ be a descriptive proximal path fibration. Then the descriptive proximal topological complexity (or simply descriptive proximal complexity) of $X$, denoted by TC$(X,\delta_{\Phi})$, is the descriptive proximal Schwarz genus of $\pi$.
\end{definition}

\begin{figure}[h]
	\centering
	\includegraphics[width=0.80\textwidth]{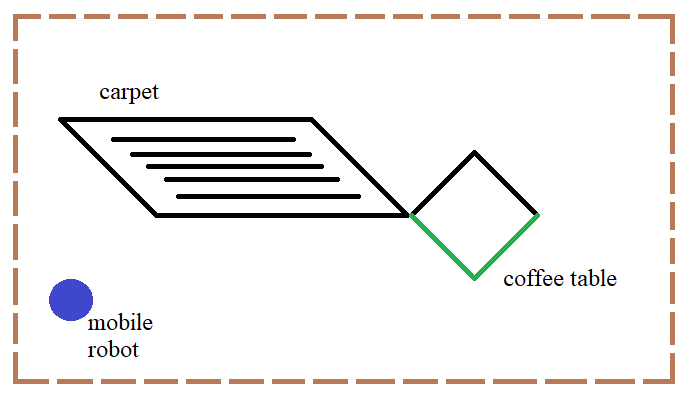}
	\caption{The proximal motion planning for cleaning the carpet and around the coffee table.}
	\label{fig:3}
\end{figure}
\begin{example}
	Consider the part of an office simulation consisting of a proximally contractible carpet and the projection of a coffee table for cleaning by a mobile robot in Figure \ref{fig:3}. Let the proximity space be representing by $(X,\delta)$. $X$ is clearly proximally connected. Since $X = W_{1} \cup W_{2}$ satisfies that there exist pc-maps $s_{1} : W_{1} \rightarrow PX$ and $s_{2} : W_{2} \rightarrow PX$ such that $\pi \circ s_{1} = 1_{W_{i}}$ and $\pi \circ s_{2} = 1_{W_{2}}$, we conclude that TC$(X,\delta) = 2$ (see Figure \ref{fig:4} for the explicit expression of $W_{1}$ and $W_{2}$). 
	\begin{figure}[h]
		\centering
		\includegraphics[width=0.80\textwidth]{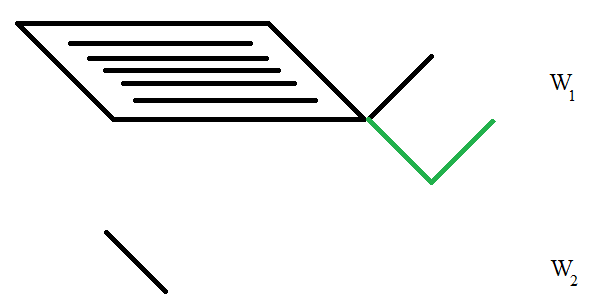}
		\caption{The proximal subspaces $W_{1}$ and $W_{2}$ that cover $X$ in Figure \ref{fig:3}.}
		\label{fig:4}
	\end{figure}

    On the other hand, if one chooses $\Phi$ as the collection of probe functions that represent colors of the office furnitures, then $(X,\delta_{\Phi})$ is not a connected descriptive proximity space. Indeed, the union of the green part and the black part is $X$ but they are not descriptively near. This shows that TC$(X,\delta_{\Phi})$ cannot be computed even if TC$(X,\delta) = 2$.
\end{example}

\begin{corollary}\label{cor2}
	For a descriptive proximity space $(X,\delta)$, TC$(X,\delta_{\Phi}) = 1$ if and only if $X$ is descriptive proximally contractible.
\end{corollary}

\begin{proof}
	See the method in the proof of Theorem \ref{teo1} by considering Definition \ref{def4}.
\end{proof}

\quad We say that a property is a descriptive proximity invariant provided that there exists a descriptive proximity isomorphism $f : (X,\delta_{\Phi}) \rightarrow (Y,\delta_{\Phi}^{'})$ such that $f$ preserves this property.

\begin{theorem}\label{teo3}
	Let $f : (X,\delta_{\Phi}) \rightarrow (Y,\delta_{\Phi}^{'})$ be a dpc-map. Then $f$ is continuous with respect to $\tau(\delta_{\Phi})$ and $\tau(\delta_{\Phi}^{'})$.
\end{theorem} 

\begin{proof}
	Let $E \subset X$ and $x \in X$. $f$ is continuous if and only if $f(cl_{\Phi}(E)) \subset cl_{\Phi^{'}}(f(E))$. Assume that $x \delta_{\Phi} E$. Then we get $x \in cl_{\Phi}(E)$. It follows that $f(x) \in f(cl_{\Phi}(E))$. On the other hand, the fact $f$ is a dpc-map implies that $f(x) \delta_{\Phi}^{'} f(E)$. This means that $f(x) \in cl_{\Phi^{'}}(f(E))$. Therefore, we find $f(cl_{\Phi}(E))$ is a subset of $cl_{\Phi^{'}}(f(E))$.
\end{proof}

\begin{corollary}
	Every topological invariant is a descriptive proximity invariant. Moreover, descriptive topological complexity is a descriptive proximity invariant.
\end{corollary}

\begin{proof}
	By Theorem \ref{teo3}, descriptive proximally isomorphic spaces are homeomorphic. This proves the first statement. For the second part, it is enough to recall that TC is a homotopy invariant, i.e., homeomorphic spaces have the same number in terms of TC. 
\end{proof}

\begin{definition}
	\textbf{i)} Assume that $(X,\delta_{\Phi})$ is a proximity space and $Y \subset X \times X$. The relative descriptive proximal topological complexity TC$_{X}(Y,\delta_{\Phi_{Y}})$ is defined as the descriptive proximal Schwarz genus of the descriptive proximal fibration $\pi^{'} : PX^{Y} \rightarrow Y$, where $PX^{Y}$ is a subset of $PX$ that consists of all descriptive proximal paths $\alpha : [0,1] \rightarrow X$ with the property $(\alpha(0),\alpha(1)) \in Y$.
	
	\textbf{ii)} Let $f : (X,\delta_{\Phi}^{1}) \rightarrow (Y,\delta_{\Phi}^{2})$ be a descriptive proximal fibration. Given a descriptive proximal fibration $\pi^{f} : (PX,\delta_{\Phi}^{'}) \rightarrow (X \times Y,\delta_{\Phi}^{''})$ with $\pi^{f} = (1 \times f) \circ \pi$, i.e., for any descriptive proximal path $\alpha$ on $X$, $\pi^{f}(\alpha) = (\alpha(0),f \circ \alpha(1))$, the descriptive proximal topological complexity of $f$, denoted by TC$_{\delta_{\Phi}^{1},\delta_{\Phi}^{2}}(f)$, is defined as genus$_{\delta_{\Phi}^{'},\delta_{\Phi}^{''}}(\pi^{f})$.
\end{definition}

\begin{definition}
	\textbf{i)} Let $f$ and $g$ be any two dpc-maps from $(X,\delta_{\Phi}^{1})$ to $(Y,\delta_{\Phi}^{2})$. Then the descriptive proximal homotopic distance between $f$ and $g$, denoted by D$_{\delta_{\Phi}^{1},\delta_{\Phi}^{2}}(f,g)$, is the minimum integer $m>0$ if the following hold:
	\begin{itemize}
		\item $X$ has at least one cover $\{V_{1}, \cdots, V_{m}\}$.
		\item For all $j = 1, \cdots, m$, $f|_{V_{j}} \simeq_{\delta_{\Phi}^{1},\delta_{\Phi}^{2}} g|_{V_{j}}$.
	\end{itemize}
	If such a covering does not exist, then D$_{\delta_{\Phi}^{1},\delta_{\Phi}^{2}}(f,g) = \infty$.
	
	\textbf{ii)} Given any dpc-maps $f_{1}, f_{2}, \cdots, f_{n}$ from $(X,\delta_{\Phi}^{1})$ to $(Y,\delta_{\Phi}^{2})$, the descriptive proximal higher ($n$th) homotopic distance between $f_{1}, f_{2}, \cdots, f_{n}$, denoted by D$_{\delta_{\Phi}^{1},\delta_{\Phi}^{2}}(f_{1},f_{2},\cdots,f_{n})$, is the minimum integer $m>0$ if the following hold:
	\begin{itemize}
		\item $X$ has at least one cover $\{V_{1}, \cdots, V_{m}\}$.
		\item For all $j = 1, \cdots, m$, $f_{1}|_{V_{j}} \simeq_{\delta_{\Phi}^{1},\delta_{\Phi}^{2}} f_{2}|_{V_{j}} \simeq_{\delta_{\Phi}^{1},\delta_{\Phi}^{2}} \cdots \simeq_{\delta_{\Phi}^{1},\delta_{\Phi}^{2}} f_{n}|_{V_{j}}$.
	\end{itemize}
	If such a covering does not exist, then D$_{\delta_{\Phi}^{1},\delta_{\Phi}^{2}}(f_{1},f_{2},\cdots,f_{n})$ is $\infty$.
	
	\textbf{iii)} Given two dpc-maps $f$, $g : (X,\delta_{\Phi}^{1}) \rightarrow (Y,\delta_{\Phi}^{2})$ with $A \subset X$, the relative descriptive proximal homotopic distance between $f$ and $g$ on $A$, denoted by D$^{X}_{\delta_{\Phi}^{1},\delta_{\Phi}^{2}}(A;f,g)$, is D$_{\delta_{\Phi}^{1},\delta_{\Phi}^{2}}(f|_{A},g|_{A})$.
\end{definition}

\begin{theorem}
	Let $f_{1},f_{2},\cdots,f_{n} : (X,\delta_{\Phi}^{1}) \rightarrow (Y,\delta_{\Phi}^{2})$ be dpc-maps. Then the following hold:
	
	\textbf{i)} Assume that $\sigma$ is any permutation of $\{1,2,\cdots,n\}$. Then
	\begin{eqnarray*}
		\text{D}_{\delta_{\Phi}^{1},\delta_{\Phi}^{2}}(f_{1},f_{2},\cdots,f_{n}) = \text{D}_{\delta_{\Phi}^{1},\delta_{\Phi}^{2}}(f_{\sigma(1)},f_{\sigma(2)},\cdots,f_{\sigma(n)}).
		\end{eqnarray*} 
		
		\textbf{ii)} D$_{\delta_{\Phi}^{1},\delta_{\Phi}^{2}}(f_{1},f_{2},\cdots,f_{n}) = 1$ if and only if $f_{1},f_{2},\cdots,f_{n}$ are descriptive proximally homotopic distance.
		
		\textbf{iii)} Assume that $1 < r < n$. Then
		\begin{eqnarray*}
			\text{D}_{\delta_{\Phi}^{1},\delta_{\Phi}^{2}}(f_{1},f_{2},\cdots,f_{r}) \leq  \text{D}_{\delta_{\Phi}^{1},\delta_{\Phi}^{2}}(f_{1},f_{2},\cdots,f_{n}).
		\end{eqnarray*}
	    
	    \textbf{iv)} Let $f_{1},f_{2},\cdots,f_{n} : (X,\delta_{\Phi}^{1}) \rightarrow (Y,\delta_{\Phi}^{2})$ and $g_{1},g_{2},\cdots,g_{n} : (X,\delta_{\Phi}^{1}) \rightarrow (Y,\delta_{\Phi}^{2})$ be dpc-maps. Then \[\text{D}_{\delta_{\Phi}^{1},\delta_{\Phi}^{2}}(f_{1},f_{2},\cdots,f_{n}) =  \text{D}_{\delta_{\Phi}^{1},\delta_{\Phi}^{2}}(g_{1},g_{2},\cdots,g_{n})\] if $f_{i} \simeq_{\delta_{\Phi}^{1},\delta_{\Phi}^{2}} g_{i}$ for each $i \in \{1,\cdots,n\}$.
\end{theorem}

\begin{proof}
	The proof is in a same way with the proofs of Theorem \ref{t6} and Theorem \ref{t7}.
\end{proof}

\begin{definition}
	\textbf{i)} Given any descriptive proximally continuous maps $f_{1},f_{2},\cdots,f_{n}$ from $(X,\delta_{\Phi}^{1})$ to $(Y,\delta_{\Phi}^{2})$ with $A \subset X$, the relative descriptive proximal higher homotopic distance between $f_{1},f_{2},\cdots,f_{n}$ on $A$, denoted by D$^{X}_{\delta_{\Phi}^{1},\delta_{\Phi}^{2}}(A;f_{1},f_{2},\cdots,f_{n})$, is D$_{\delta_{\Phi}^{1},\delta_{\Phi}^{2}}(f_{1}|_{A},f_{2}|_{A},\cdots,f_{n}|_{A})$.
	
	\textbf{ii)} Given a descriptive proximal fibration $f : (X,\delta_{\Phi}^{1}) \rightarrow (Y,\delta_{\Phi}^{2})$ and the projection maps $\pi_{1} : (X \times Y,\delta_{\Phi}) \rightarrow (X,\delta_{\Phi}^{1})$, $\pi_{2} : (X \times Y,\delta_{\Phi}) \rightarrow (Y,\delta_{\Phi}^{2})$, the descriptive proximal topological complexity of $f$ is defined as TC$_{\delta_{\Phi}^{1},\delta_{\Phi}^{2}}(f) =$ D$_{\delta_{\Phi},\delta_{\Phi}^{2}}(f \circ \pi_{1},\pi_{2})$.
	
	\textbf{iii)} Let $n>1$. Assume that $f = (f_{1},f_{2},\cdots,f_{n}) : (X,\delta_{\Phi}^{1}) \rightarrow (Y^{n},\delta_{\Phi}^{'})$ is a descriptive proximal fibration, and $\pi_{i} : (X^{n},\delta_{\Phi}) \rightarrow (X,\delta_{\Phi}^{1})$ is the $i$th projection map for each $i \in \{1,\cdots,n\}$. Then the descriptive proximal higher topological complexity of $f$ is defined by TC$_{n,\delta_{\Phi}^{1},\delta_{\Phi}^{'}}(f) =$ D$_{\delta_{\Phi},\delta_{\Phi}^{'}}(f \circ \pi_{1},f \circ \pi_{2},\cdots,f \circ \pi_{n})$.
\end{definition}

\begin{theorem}
	\textbf{i)} The descriptive proximal topological complexity of $(X,\delta_{\Phi})$ is D$_{\delta_{\Phi}^{'},\delta_{\Phi}}(p_{1},p_{2})$ for the projection maps $p_{1}, p_{2} : (X \times X,\delta_{\Phi}^{'}) \rightarrow (X,\delta_{\Phi})$.
	
	\textbf{ii)} The descriptive proximal higher topological complexity of $(X,\delta_{\Phi})$ is \[\text{D}_{\delta_{\Phi}^{'},\delta_{\Phi}}(p_{1},p_{2},\cdots,p_{n})\] for the $i$th projection map $p_{i} : (X \times X,\delta_{\Phi}^{'}) \rightarrow (X,\delta_{\Phi})$ for each $i \in \{1,\cdots,n\}$.
\end{theorem}

\begin{proof}
	\textbf{i)} See the method in the proof of Corollary \ref{c2}.
	
	\textbf{ii)} It is the generalization of the part \textbf{i)}.
\end{proof}

\begin{example}
	Consider Figure \ref{fig:1} again. If a set of probe functions $\phi : X \rightarrow \mathbb{R}$ represents the color in $X$, that is the projection of the lounge of a house on the floor, then the fact that $(X,\delta_{\Phi})$ is not connected implies that the descriptive proximal TC of $X$ cannot be computed. On the other hand, assume that a set of probe functions $\phi : X \rightarrow \mathbb{R}$ represents the feature that any furniture has at least one edge. Then $X$ is $\delta_{\Phi}-$connected right now. The descriptive proximal TC of sofa1 is $2$ similar to the method in Example \ref{ex1} (see Figure \ref{fig:2} for the method again). It can be easily seen the fact TC$_{3}(\text{sofa1},\delta_{\Phi}) = 2$ by using the same method. Moreover, the computation holds true for other furnitures.
\end{example}

\section{Conclusion}
\label{sec:conclusion}
\quad In the literature, proximal motion planning problem studies follow motion planning problem studies on topological spaces and takes these algorithmic works to a different ground. The strength of the paper comes from the fact that it is the basis for further investigations of motion planning problems in proximity spaces. Future research will not be surprised to find a significant connection between the descriptive proximity spaces, which support many feature vector types such as color or shape, and the parametrized topological complexity, which places a high value on external conditions. In \cite{PetersTane2:2022}, Peters and Vergili present the Lusternik-Schnirelmann category (LS-cat) for proximity (and descriptive proximity) spaces. The relationship between LS-cat and TC can be investigated as an open problem. As another issue, the relation of other topological complexity types, such as symmetric topological complexity or monoidal topological complexity, with different proximity types such as Lodato proximity or Cech proximity also remains an open problem right now. In summary, this is an investigation that examines only the basic concepts of TC computation in proximity or descriptive proximity spaces.

\acknowledgment{The Scientific and Technological Research Council of Turkey TÜBTAK-1002-A funded this work under project number 122F454. The second author is thankful to the Azerbaijan State Agrarian University for all their hospitality and generosity during his stay.}

\end{document}